\documentclass[12pt]{amsart}
\usepackage{enumerate}
\usepackage{amsthm}
\usepackage{amsmath,amssymb,latexsym,amscd}
\usepackage{tikz}
\makeatletter
\@namedef{subjclassname@2020}{%
	\textup{2020} Mathematics Subject Classification}
\makeatother
\theoremstyle{definition}
\newtheorem{theorem}{Theorem}[section]
\newtheorem{lemma}[theorem]{Lemma}
\newtheorem{proposition}[theorem]{Proposition}

\theoremstyle{definition}
\theoremstyle{definition}

\newtheorem{remark}[theorem]{Remark}
\newtheorem{example}[theorem]{Example}
\usepackage{indentfirst}

\begin{document}
	\baselineskip=16.5pt
	\title[]{Fixed Point Sets and Orbit Spaces of  Wedge of three Spheres}
	\author[Dimpi and Hemant Kumar Singh]{ Dimpi and Hemant Kumar Singh}
	\address{ Dimpi \newline 
		\indent Department of Mathematics\indent \newline\indent University of Delhi\newline\indent 
		Delhi -- 110007, India.}
	\email{dimpipaul2@gmail.com}
	\address{  Hemant Kumar Singh\newline\indent 
		Department of Mathematics\newline\indent University of Delhi\newline\indent 
		Delhi -- 110007, India.}
	\email{hemantksingh@maths.du.ac.in}

	\date{}
	\thanks{ The first author of the paper is  supported by JRF of UGC, New Delhi, with  reference no.: 201610039267.}
	\begin{abstract} 
		
		Let $X$ be a finite  CW-complex having mod $p$  cohomology isomorphic to a wedge of three  spheres $\mathbb{S}^n\vee \mathbb{S}^m \vee \mathbb{S}^l,~ 1\leq n \leq m \leq l.$ The aim of this  paper is to determine the fixed point sets of actions of the cyclic group of prime order   on $X.$ We also  classify the orbit spaces of free actions of $G=\mathbb{Z}_p, p$ a prime or $G=\mathbb{S}^d,~d=1,3,$ on $X$  and derive the Borsuk-Ulam type results.

	\end{abstract}
	\subjclass[2020]{Primary 57S17; Secondary 55R20}
	
	\keywords{Fixed Point Sets; Orbit Spaces;  Fibration; Totally nonhomologous to zero; Leray-Serre spectral sequence.}

	\maketitle
	\section {Introduction}
	\noindent Let $G$ be a topological group acting on a topological space $X.$ The fixed point set  and the orbit space  are two new spaces associated to the transformation group $(G,X).$ It has been an interesting problem to study the cohomological structure of the fixed point sets.  Smith \cite{s} proved that the fixed point sets of homeomorphisms of prime order on a finite  dimensional polyhedron $X$ having the mod $p$ cohomology $n$-sphere (respectively, $n$-disc)  are  mod $p$  cohomology $r$-sphere (respectively, $r$-disc) where $-1 \leq r \leq n$ (respectively, $0 \leq r \leq n$). Su \cite{j1} determined the fixed point sets of a cyclic group of prime order on a space $X$ with the mod $p$ cohomology  isomorphic to the product of spheres $\mathbb{S}^m \times \mathbb{S}^n.$  Bredon \cite{Bredon}  gave an alternate proof of J. C. Su result by proving that if a finite CW-complex  $X$ satisfies poincar\'{e} duality  then each component of the fixed point set also satisfies poincar\'{e}  duality. Chang et al. \cite{c} discussed the fixed point sets of  actions $G=\mathbb{S}^1$ on the product of two spheres. 
	Dolzel et al. \cite{f} and Singh \cite{singh} discussed the similar questions for $G=\mathbb{Z}_p,$ $p$ a prime or $G=\mathbb{S}^1$ actions on a particular wedge of three spheres $\mathbb{S}^n \vee \mathbb{S}^{2n} \vee \mathbb{S}^{3n}.$ In this paper, we contribute to this question by generalizing it for arbitrary  wedge of three spheres $\mathbb{S}^n \vee \mathbb{S}^m \vee \mathbb{S}^l.$ An another thread of research is to classify the orbit spaces when $G$ acts freely on $X.$ First such question was raised by H. Hopf to classify the orbit spaces of free actions of finite cyclic groups on spheres. It is well known that projective spaces $P^n(q),$ where $q=1,2$ and $4,$ are the orbit spaces of free actions of $\mathbb{Z}_2,~ \mathbb{S}^1$ and $\mathbb{S}^3$ on  $\mathbb{S}^n,$ $\mathbb{S}^{2n+1}$ and $\mathbb{S}^{4n+3},$ respectively. Recently, Dey et al. \cite{dey} and Morita et al. \cite{Mattos} determined the orbit spaces of free involutions on real Milnor manifolds and Dold manifolds, respectively. The orbit spaces of free actions of $G=\mathbb{Z}_p, p$ a prime, or $G=\mathbb{S}^d, d=1$ or $3,$ on the cohomology product of spheres $\mathbb{S}^n \times \mathbb{S}^m$ have been studied in \cite{orbit space, anju}. A similar question for a particular wedge of three spheres $\mathbb{S}^n \vee \mathbb{S}^{2n} \vee \mathbb{S}^{3n}$ for $G=\mathbb{Z}_p, p$ a prime, or $G=\mathbb{S}^1$ actions have been answered in \cite{o, hemant sir}. In this  paper,  we  generalize these results. We classify the  orbit spaces of free actions of $G= \mathbb{Z}_p, ~p$ a prime or $G=\mathbb{S}^d,~d=1,3$ on a finite CW-complex having homotopy type of wedge of three spheres $\mathbb{S}^n \vee \mathbb{S}^{m} \vee \mathbb{S}^{l}.$  As an application we derive the Borsuk-Ulam type results.

	\section{Preliminaries}  
	\noindent  We will recall  some known facts that will be used in this paper.  Let $G$ be a compact Lie group. Then there exists a contractible paracompact space $E_G$ on which $G$ acts freely  and the orbit space $B_G$ of this action  is called the classifying space of $G.$ Let $G$ act on a finite CW-complex $X.$ Then the associated  fibration $ X \stackrel{i} \hookrightarrow X_G \stackrel{\pi} \rightarrow B_G$ of this action is called Borel fibration where  $X_G$ is the orbit space $(X\times E_G)/G$ obtained by the diagonal action of  $G.$ 
	\noindent
	For the $E_2$-term of the Leray-Serre spectral sequence associated to the Borel
	fibration $ X \stackrel{i} \hookrightarrow X_G \stackrel{\pi} \rightarrow B_G,$  we have the following results: 
	
	\begin{proposition}(\cite{mac})
		Let $G$ act on a finite CW-complex $X$ and $\pi_1(B _G)$ acts trivially on $H^*(X).$ Then, the system of local coefficients on $B_G$ is simple and $E^{k,i}_2= H^{k}(B_G; H^i(X)), ~k,i \geq 0.$
	\end{proposition}

	\begin{proposition}(\cite{bredon})\label{re} 
		Let $G=\mathbb{Z}_p$ act on a finite CW-complex $X$ and $\pi_1(B _G)$ acts nontrivially on $H^*(X).$  Then, 
		\[
		E_2^{k,i}=
		\begin{cases}
			\text{$ker~ \tau$} & \mbox{for}~ ~k=0 \\
			\text{$ker ~\tau$/$im~ \sigma$} & \text{for } k>0~ \mbox{even}\\
			\text{$ker ~\sigma $/$im ~\tau$} & \text{for } k>0~ \mbox{odd}\\
			
		\end{cases}
		\]
		where $\tau =1-g^*, $ $  \sigma = 1+g^*+g^{*2}+\cdots +g^{*p-1}$ and $g^*$ is induced by a generator $g$ of $G.$
	\end{proposition}
	\begin{proposition}(\cite{mac})
		Suppose that  the system of local coefficients on $B_G$ is simple.	Then the  homomorphism $i^*: H^*(X_G) \rightarrow H^*(X)$ and $\pi^*: H^*(B_G) \rightarrow H^*(X_G)$ are the edge homomorphisms, \begin{center}
			$ H^k(B_G)=E_2^{k,0}\rightarrow E_3^{k,0}\rightarrow \cdots E_k^{k,0}\rightarrow E_{k+1}^{k,0} = E_{\infty}^{k,0} \subset H^k(X_G)$ and $  H^i(X_G) \rightarrow E_{\infty}^{0,i} \hookrightarrow E_{l+1}^{0,i} \hookrightarrow E_{l}^{0,i} \hookrightarrow \cdots \hookrightarrow E_{2}^{0,i} \hookrightarrow E_2^{0,i} \hookrightarrow H^i(X),$ respectively.
		\end{center} 
	\end{proposition}
	\noindent For the proofs about the results of spectral sequences, we refer the reader to \cite{mac}. 
	\begin{proposition}(\cite{bredon})\label{prop 4.5}
		Let $G=\mathbb{Z}_p,~ p$ a prime or $G=\mathbb{S}^1$ acts freely on a finite CW-complex $X.$ If $H^i(X;R)=0~ \forall~ i>n$ then $H^i(X/G;R)=0~ \forall~ i>n,$ where $R=\mathbb{Z}_p,~p$ a prime or $\mathbb{Q},$  respectively. 
	\end{proposition}
	\noindent A similar result also holds for  $G=\mathbb{S}^3$ actions on $X$ with coefficients in  $\mathbb{Q},$ if the associated sphere bundle $G \hookrightarrow X \rightarrow X/G$ is orientable \cite{kaur}.
	
	\begin{proposition}(\cite{tom})\label{tom}
		Let $G$ be a compact Lie group act freely on a paracompact space $X.$ Then the space $X_G$ is homotopy equivalent to the orbit space $X/G.$
	\end{proposition}
	\begin{proposition} (\cite{allen})
		Let $\mathbb{S}^{n-1} \hookrightarrow E \stackrel{\pi}\rightarrow B$ be an oriented $(n-1)$-sphere bundle. Then we have a long exact sequence
		$$\cdots H^i(B) \rightarrow H^i(E) \rightarrow H^{i-n+1}(B) \stackrel{\cup }\rightarrow H^{i+1}(B) \cdots$$ which begin with  $$ 0 \rightarrow H^{n-1}(B) \rightarrow H^{n-1}(E) \rightarrow H^0(B) \rightarrow H^n(B) \cdots$$ 
		where $\cup :H^i(B) \rightarrow H^{i+n}(B)$ maps $z \mapsto z\cup u,$ $u \in H^{n}(B)$ denotes the characterstic class of the sphere bundle. This sequence is called the \it{Gysin sequence}.
	\end{proposition}
	Recall that a space $X$ is said to be  totally nonhomologous to zero (TNHZ) in $X_G$ if the inclusion map $X \hookrightarrow X_G$ induces
	a surjective homomorphism $H^*(X_ G ; R) \rightarrow H^*(X;R),$ where $R=\mathbb{Z}_p,~p$ a prime or $\mathbb{Q}.$  If the fixed point set $F$ of an action of $G$ on $X$ is nonempty then  the following are well known results:
	\begin{proposition}(\cite{bredon})\label{ pro 1}
		Let $G= \mathbb{Z}_p,~ p$ a prime, act on a finite CW-complex $X.$ Then, \begin{center}
			$ \sum_{i \geq j}$ rk $H^{i}(F;\mathbb{Z}_p) \leq \sum_{i \geq j} $ rk $H^{i}(X;\mathbb{Z}_p),$ for each $j \geq 0.$
			
		\end{center} 
	\end{proposition}
	
	\begin{proposition}(\cite{bredon}) \label{pro4}
		Let $G=\mathbb{Z}_p,~ p$ a prime, act on a finite CW-complex $X.$ Suppose $\sum$ rk $H^i(X;\mathbb{Z}_p)< \infty.$ Then the following statements are equivalent: 
		\begin{enumerate}
			\item $X$ is TNHZ (mod $p$) in $X_G.$  
			\item $\sum$ rk $H^i(F;\mathbb{Z}_p) =\sum$ rk $H^i(X;\mathbb{Z}_p).$ 
			\item $G$ acts trivially on $H^*(X; \mathbb{Z}_p)$ and spectral sequence $E^{k,i}_2$ of $X_G \rightarrow B_G$ degenerates.
		\end{enumerate} 
	\end{proposition}
	\begin{proposition}(\cite{bredon}) \label{floyd}
		Let $G=\mathbb{Z}_p,~p$ a prime, act on  a finite CW-complex $X.$ If $\sum$ rk $H^i(X;\mathbb{Z}_p)< \infty$ then by the Floyd's formula we have  $\chi(X)\equiv\chi(F)(mod~p),$ where the Euler characteristic is defined in terms of the mod$~ p$ \v{C}ech cohomology. 
	\end{proposition}
	
	\noindent Recall  that 
	$H^*(\mathbb{S}^n \vee \mathbb{S}^m \vee \mathbb{S}^l; R)=R[a,b,c]/<{a^2,b^2,c^2,ab,bc,ac}>,$ where deg $a=n$, deg $b=m$ and deg $c=l,$ where $R=\mathbb{Z}_p,$ $p$ a prime or $\mathbb{Q}.$\\
	\indent Throughout the paper, $H^*(X)$ will denote the \v{C}ech cohomology of a space $X$ and   $X\sim_R Y,$ means $H^*(X; R)\cong H^*(Y;R),$ where  $R = \mathbb{Z}_p , ~p$ a prime or  $\mathbb{Q}.$

	\section{Fixed Point Sets of Wedge of Three Spheres}                               
	
	\noindent Let  $G=\mathbb{Z}_p,$ $p>2$ a prime, act on a space $X$ having mod $p$ cohomology isomorphic to a wedge of three spheres $\mathbb{S}^n \vee \mathbb{S}^m \vee \mathbb{S}^l,$ $1\leq n \leq m \leq l.$ If $X$ is TNHZ in $X_{G}$ then it is easy to see that the fixed point sets  are similar to that of $X \sim_{\mathbb{Z}_p} \mathbb{S}^n \vee \mathbb{S}^{2n} \vee \mathbb{S}^{3n}$ \cite{f, monika}. If $X$ is not TNHZ in $X_G$ and $X \sim_{\mathbb{Z}_p} \mathbb{S}^n \vee \mathbb{S}^{2n} \vee \mathbb{S}^{3n}$ then the possibilities of the fixed point sets are odd dimensional spheres \cite{o}. In this section, we  generalize this case for arbitrary wedge of three spheres. 
	\begin{theorem}
		Let $G=\mathbb{Z}_p$, $p>2$  a prime, act on a finite CW-complex $X \sim_{\mathbb{Z}_p} \mathbb{S}^n \vee \mathbb{S}^m \vee \mathbb{S}^{l}$ , $1 \leq n\leq m \leq l,$ where $X$ is not TNHZ in $X_{G}.$ If the fixed point set $F$ is  nonempty then $F$ must be one of the following:
		\begin{enumerate}
			\item $F \sim_{\mathbb{Z}_p} S^q,$~~$0\leq q \leq l.$
			\item $F \sim_{\mathbb{Z}_p} S^{q} \vee S^{r},$ $p=3$ or $5$, $0 \leq q \leq m$, $0\leq q \leq r <l,$ and either  $n \leq m =l$ or $n=m \leq l.$
			\item $F\sim_{\mathbb{Z}_p} P^2(q),$ $p=3$ or $5,$ $~ q $ even, $q<m,$ and either  $n \leq m =l$ or $n=m \leq l.$
			\item $F \sim_{\mathbb{Z}_p} pt$, $p=3,$ and  $n=m =l.$
		\end{enumerate}
	\end{theorem} 
	\begin{proof}
		As $X$ is not TNHZ	in $X_G$ and $F$ is nonempty, we have $1 \leq \sum_i$ rk $H^i(F,\mathbb{Z}_p)\leq 3.$ \\
		First, we assume that   $G$ act trivially on $H^*(X).$ By Proposition \ref{pro4}, the spectral sequence of Borel fibration $X \hookrightarrow X_G \rightarrow B_G$  does not degenerate. So, the differential $d_r$ must be nontrivial for some $r\geq 2.$  By the multiplicative property of the spectral sequence and the fact that $F$ is nonempty, the differentials $d_{n+1},$ $d_{m+1}$ and $d_{l+1}$ are always trivial. The only possibilities for nontrivial differentials are $d_{m-n+1},$ $d_{l-m+1}$ or $d_{l-n+1}.$ If the differential  $d_{m-n+1}$ is nontrivial then  $d_{m-n+1}: E_{m-n+1}^{k,m}\rightarrow E_{m-n+1}^{k+m-n+1,n}$ are isomorphisms for all $k \geq 0.$ Clearly, $E_{m-n+2}^{*,*}= E_{\infty}^{*,*}.$ Note that $j^*: H^k(X_G) \rightarrow H^k(F_G)$ are isomorpisms for all $k>l$ \cite{bredon}, we get rk $H^*(F)=2.$ Thus   $F \sim_{\mathbb{Z}_p} \mathbb{S}^q$ for some $q\geq 0.$ By Proposition \ref{ pro 1}, we get $0 \leq q \leq l.$ By the Floyd's formula, we get  $\chi(F)= \chi(X)=0$ or $2.$ This implies that if  $q$ odd (respectively, even) then exactly two of  $n,m $ or $l$ are odd (respectively, even).  We will get same conclusion for nontrivial differentials $d_{l-m+1}$ or $d_{l-n+1}.$  This realizes possibility (1).
		
		\indent 	Next, assume that $G$ act nontrivially on $H^*(X).$ In this case, either  $n=m \leq l$ or $n \leq m =l.$ \\ \indent First, we consider $n=m<l.$ Let $G=<g>$ and  $\tau = 1-g^*,$ where $g^*$ is a nontrivial homomorphism induced by $g.$ Then, $\tau^p = 0$ and hence $Ker~ \tau \cong \mathbb{Z}_p.$ Suppose $Ker~ \tau =<a>.$ Thus $g^*(a)=a,$ where $\{a,b\}$ is  basis of $H^n(X).$ As ${g^*}^p=1,$ we get $g^*(b)=\alpha a+b,$ where $0 \neq \alpha \in \mathbb{Z}_p.$ Let $\sigma = 1 +g^*+g^{*2}+ \cdots + g^{*p-1}.$ Clearly $\sigma =0,$ and by Proposition \ref{re}, we have $E_2^{k,n} \cong \mathbb{Z}_p~ \forall ~k \geq 0.$  We also have  $E_2^{k,i}\cong \mathbb{Z}_p$ for $i=0,l;$ $k \geq 0.$ 
		By the multiplicative property of the spectral sequence, we get $d_{n+1}=d_{l+1}=0.$  If  $d_{l-n+1}=0$ then  $E_2^{*,*}=E_{\infty}^{*,*}.$  Thus,  rk $H^*(F)=3.$ We get $\chi(F)=3,1$ or $-1.$ As $\chi(X)\equiv \chi(F)$(mod~$p$),  $p$ must be $3$ or $5.$ Note that, for $p=5,$ both $n$ and $l$ are either even or odd. Suppose that $F$ is connected and $0 < q \leq r $ are non vanishing dimensions  of $F$ with generators  $u$ and $v,$ respectively. Clearly,   $uv=v^2=0.$ If $u^2 \neq 0$ then $F \sim_{\mathbb{Z}_p} P^2(q).$ Note that $q$ must be even \cite{bredon}, and  by Proposition \ref{ pro 1},  we have  $q<n.$ In this case, $n$ is odd and $l$ is even for $p=3,$ and both $n$ and $l$ are odd for $p=5.$  If $u^2=0$ then   $F \sim_{\mathbb{Z}_p} \mathbb{S}^q \vee \mathbb{S}^r,$ where $0 < q \leq m$ and $0< q \leq r \leq l.$ Next, if $F$ is disconnected then $F \sim_{\mathbb{Z}_p} \mathbb{S}^q \vee \mathbb{S}^r$ where either $q=0$ or $q=r=0.$ Further, if  $d_{l-n+1}\neq 0$ then  rk $H^*(F)=1$ or $2,$ according as $l-n$ is odd or even which contradicts the Floyd's formula. We get similar possibilities for $n<m=l.$ 
		
		\indent  Now, we consider  $n=m=l.$ Then $\{a,b,c\}$ is basis of $H^n(X).$  As above, we get either  $ker ~\tau \cong \mathbb{Z}_p$ or $ker ~\tau \cong \mathbb{Z}_p \oplus \mathbb{Z}_p.$ If   $ker~ \tau \cong \mathbb{Z}_p \oplus \mathbb{Z}_p$ then we  can define $g^*(a)=a,$ $g^*(b)=b$ and $g^*(c)=c+\alpha a+\beta b,$ where $  g^*(c)\neq c ~\&~$$\alpha,\beta \in \mathbb{Z}_p.$ This implies that  $\sigma =0,$ and $E_2^{k,n}=  \mathbb{Z}_p \oplus \mathbb{Z}_p~\forall~k \geq 0.$ It is easy to see that rk $H^*(F)=3.$ As  in the case when $n=m<l,$ we obtain $F \sim_{\mathbb{Z}_p} P^2(q)$ for $p=5,$ or  $F \sim_{\mathbb{Z}_p} \mathbb{S}^q \vee \mathbb{S}^r,$ for $p=3$ or $5.$  If $Ker ~\tau \cong \mathbb{Z}_p,$ then  $im ~\sigma$ either $0$ or $\mathbb{Z}_p.$ For example, if we take  $g^*(a)= a, g^*(b)=a+b$ and $g^*(c)=a+c$ then $im~\sigma =0$ and if we take  $g^*(a)= a, g^*(b)=a+b~\&~$$g^*(c)=a+b+c$ then $im~\sigma = Ker ~ \tau.$ If  $im~ \sigma =0$ then $E_2^{k,n}=  \mathbb{Z}_p~ \forall~ k \geq 0.$  Consequently,  $E_2^{*,*}=E_{\infty}^{*,*},$ and hence  rk $H^*(F)=2.$ Thus, $F \sim_{\mathbb{Z}_p} \mathbb{S}^q, ~ 0 \leq q \leq n.$ If $im ~ \sigma= \mathbb{Z}_p$ then $E_2^{k,n}=0~ \forall~ k >0$ and hence rk $H^*(F)=1.$ So, $F\sim_{\mathbb{Z}_p} pt.$ By the Floyd's formula, we must have $p=3.$   This completes the proof.
	\end{proof} 
	\begin{remark}
		For the actions of $G=\mathbb{Z}_2$ or $G= \mathbb{S}^1,$  possibilities of the fixed point sets of $X\sim_{R} \mathbb{S}^n \vee \mathbb{S}^m \vee \mathbb{S}^l,$ $1 \leq n\leq m\leq l,$ are similar to that of  $X\sim_R \mathbb{S}^n \vee \mathbb{S}^{2n}\vee \mathbb{S}^{3n}$ where $R=\mathbb{Z}_2$ or $\mathbb{Q},$ respectively \cite{singh, m circle}.
	\end{remark}
	
	\begin{example}
		It is shown that if $X$ is not TNHZ then there exists  an action of $G=\mathbb{Z}_p$  on  $Y \simeq \mathbb{S}^2 \vee \mathbb{S}^{n+2}$ with contractible fixed point set \cite{o}. We have an action of  $G$  on $S^l$, where $l\geq 1,$ with the fixed point set $\mathbb{S}^q,$ where $l-q$ is even. This action induces an action of $G=\mathbb{Z}_p$ on the wedge sum $Y \vee \mathbb{S}^l \sim_{\mathbb{Z}_p} \mathbb{S}^2 \vee \mathbb{S}^{n+2} \vee \mathbb{S}^l $ with the fixed point set $F\sim_{\mathbb{Z}_p} \mathbb{S}^q.$ This realizes possibility (1).
	\end{example}
	\begin{example}
		Consider an action of $G=\mathbb{Z}_3$ on $ \mathbb{S}^n \times \mathbb{S}^n,n=1,3,7$ with the fixed point set  $F \sim_{\mathbb{Z}_3} \{pt\} + \mathbb{S}^q$ where $q=0,2,6$ \cite{bredon}. This gives an action of $G$ on $Y= (\mathbb{S}^n \times \mathbb{S}^n) - \{pt\} \simeq \mathbb{S}^n \vee \mathbb{S}^n$ with  $F\sim_{\mathbb{Z}_3} \mathbb{S}^{q}.$ Take an action of $G$ on $\mathbb{S}^{l}$ with $F \sim_{\mathbb{Z}_3} \mathbb{S}^{r},$ where $l-r$ is even. So, we obtain an action of $G$ on $\mathbb{S}^n \vee \mathbb{S}^n \vee \mathbb{S}^{l}$  with  $F \sim_{\mathbb{Z}_3} \mathbb{S}^q \vee \mathbb{S}^{r}.$ This realizes possibility (2).
	\end{example}
	\begin{example}
		
		By the construction of Bredon \cite{bredon}, we have realized the possibility (3). We have an action of $G=\mathbb{Z}_3$ on $Y \simeq \mathbb{S}^n \vee \mathbb{S}^n, n=3,7$ with  $F\sim_{\mathbb{Z}_3} \mathbb{S}^{q}, q=2,6.$ Let $G$ act on $\mathbb{S}^5$ with the fixed point set $\mathbb{S}^3$ and  consider the Hopf map $\eta: \mathbb{S}^3 \rightarrow \mathbb{S}^2 \subset Y \simeq \mathbb{S}^3 \vee \mathbb{S}^3.$ As $\mathbb{S}^2$ is homotopically trivial in $Y,$  $\eta$  can be extended to an equivariant map  $\varphi: \mathbb{S}^5 \rightarrow  Y .$ Now, consider an action of $G$ on $\mathbb{D}^6$ with the fixed point set $\mathbb{D}^4.$ It is easy to see that  $G$ acts on  adjunction space $X=Y \cup_{\varphi} D^6 $ with the fixed point set $F = \mathbb{S}^2 \cup_{\eta} D^4 \simeq \mathbb{C}P^2.$ For a suitable choice of orientation on $\mathbb{S}^5,$ we get $[\varphi]=0$ in $\pi_5(Y),$ and hence $X\simeq  \mathbb{S}^{3}\vee \mathbb{S}^{3} \vee \mathbb{S}^{6}.$\\
		\indent Now, choose a map $\eta: \mathbb{S}^{11} \rightarrow \mathbb{S}^6 \subset Y \simeq \mathbb{S}^7 \vee \mathbb{S}^7 $ with Hopf invariant 2.  Clearly, $\eta$ can be  extended equivariantly to a map $\varphi: \mathbb{S}^{13} \rightarrow Y.$ Similarly, we  get an action of $G$ on $X=Y \cup_{\varphi} D^{14}$ with the fixed set $F= \mathbb{S}^6 \cup_{\eta} \mathbb{D}^{12}.$ For a suitable choice of orientation on $\mathbb{S}^{13},$ we get $[\varphi]=0,$ and hence $ X \simeq \mathbb{S}^7 \vee \mathbb{S}^7 \vee \mathbb{S}^{14}.$ As $\eta$ has Hopf invariant 2, we get $F \sim_{\mathbb{Z}_3} P^2(6).$ 
	\end{example}
	
	\indent  The possibilities of the fixed point sets of actions of $G=\mathbb{Z}_p$, $p$ a prime or $G=\mathbb{S}^1$  on the cohomology product of two spheres $\mathbb{S}^n \times \mathbb{S}^m$ have been discussed in  \cite{bredon, c, j1}. The fixed point sets of $G= \mathbb{Z}_p$ or $G=\mathbb{S}^1$ actions on a finite CW-complex having mod $p$ or rational cohomology  wedge of two spheres $\mathbb{S}^n \vee \mathbb{S}^m$ can be  derived by analysing the possibilities of rank of $H^*(F)$ depending on whether $X$ is TNHZ or not TNHZ in $X_G$. We have the following results:
	\begin{theorem}\label{thm 1}
		Let $G=\mathbb{Z}_p,~p$~a prime, act on a finite  CW-complex $X \sim_{\mathbb{Z}_p} \mathbb{S}^n \vee \mathbb{S}^m,$ $1 \leq n\leq m.$  Then the  fixed point set $F$ is nonempty and must be one of the following:
		\begin{enumerate}
			\item $F$ has three components and  each has point cohomology,  both $m$ and $n$ are even for $p>2$ a prime.
			\item $F\sim_{\mathbb{Z}_p} pt + \mathbb{S}^q,~~ 0<q \leq m,$ for $p>2$ a prime,  if $q$ even then both $n$ and $m$ are even and if $q$ odd then either $n$ or $m$ is even.  
			\item  $F \sim_{\mathbb{Z}_p} \mathbb{S}^{q_1} \vee \mathbb{S}^{q_{2}}$,~~ $1 \leq q_1 \leq n,\, \, 1\leq q_2 \leq m,$  for $p>2$ a prime, if $q_1$ and $q_2$ are even (respectively, odd) then both  $m$ and $n$ are  even (respectively, odd) and if one of  $q_1$ or $q_2$ is even then one of $m$ or $n$ is even.
			\item $F \sim_{\mathbb{Z}_p} P^2(q), ~~ q<n,$ if $p>2$ a prime, then all $q,m,n $ must be  even and if $p=2,$ then $q=1,2,4,8.$ 
			\item $F\sim_{\mathbb{Z}_p} pt, $  for $p>2$ a prime, either $n $ or $m$ is even.
			\item $F\sim_{\mathbb{Z}_p} \mathbb{S}^q ,p=3,$ $q \leq m,$ if  $q$ is  even then $n =m$  odd, and if $q$ is  odd then $n =m$  even.
		\end{enumerate}
	\end{theorem}
	\begin{theorem}
		Let $G=\mathbb{S}^1$ act on a finite CW-complex $X \sim_\mathbb{Q} \mathbb{S}^n \vee \mathbb{S}^m,$ $1 \leq n\leq m.$ Then the  fixed point set $F$ is nonempty and must be one of the following:
		\begin{enumerate}
			\item $F$ has three components each component has point cohomology, $m$ and $n$ both even.
			\item $F\sim_\mathbb{Q} pt + \mathbb{S}^q, ~q>0,$ if $q$ is even then   both $n$ and $m$ are even, and if $q$ is odd then either $m$ or $n$ is odd.
			\item $F \sim_\mathbb{Q} \mathbb{S}^{q_1} \vee \mathbb{S}^{q_2}$,~~$1 \leq q_1 \leq n,\, \, 1\leq q_2 \leq m,$ if both $q_1$ and $q_2$ are even (respectively, odd) then both  $m$ and $n$ are  even (respectively, odd) and if one of  $q_1$ or $q_2$ is even then one of $m$ or $n$ is even.
			\item $F \sim_\mathbb{Q} P^2(q),~q \leq n,$ if $q$ is even then both $n~ \&~m$ are even and if $q$ is odd then exactly one of $n$ or $m$ is even. 
			\item $F \sim_\mathbb{Q} pt,$ either $n$ or $m$ is even.
		\end{enumerate}
	\end{theorem}

	\noindent Examples of  these possibilities can be realized using the fact that $\mathbb{S}^n \times \mathbb{S}^m - \{pt\} \simeq \mathbb{S}^n \vee \mathbb{S}^m.$

	\section{Cohomology Algebra of The Orbit Spaces}
	\noindent In this section, we determine the cohomology algebra of orbit spaces of free actions of $G=\mathbb{Z}_p,~ p$ a prime or $G=\mathbb{S}^d,~d=1,3$ on a finite CW-complex $X \sim_R \mathbb{S}^n \vee \mathbb{S}^m \vee \mathbb{S}^l, 1\leq n\leq m \leq l,$ where $R=\mathbb{Z}_p, ~p $ a prime or $\mathbb{Q}$ respectively. \\
	\indent First, we give examples of   free actions of $G$ on $X.$
	\begin{example}{\label{example 4.1}}
		Let  $G=\mathbb{S}^3$ act on $\mathbb{S}^{4m+3}\subset \mathbb{H}^{m+1}$ by componentwize multiplication  $(g,(z_0,z_1,\cdots,z_m)) \mapsto (gz_0,gz_1,\cdots,gz_m)~ \forall g \in \mathbb{S}^{3}  .$ Take similar action of $G$ on  $\mathbb{S}^{4l+3}\subset \mathbb{H}^{l+1}$ in which a copy of $(4n+3)$-sphere  $\mathbb{S}^{4n+3}$ is common in both   spheres $\mathbb{S}^{4m+3}$ $~\&~$  $\mathbb{S}^{4l+3},$ where $1\leq n < m < l.$  After attaching  spheres $\mathbb{S}^{4m+3}$ and $\mathbb{S}^{4l+3}$ along  the identity map $i:\mathbb{S}^{4n+3} \rightarrow \mathbb{S}^{4n+3},$ we obtain a Hausdorff space $Y = \mathbb{S}^{4m+3} \cup_{\mathbb{S}^{4n+3}} \mathbb{S}^{4l+3}.$ As  $\mathbb{S}^{4n+3}$ is invariant under  actions of $G$ on $\mathbb{S}^{4m+3}$ and $\mathbb{S}^{4l+3},$ $G$ acts freely on $Y.$ Define $U=Y-\{p_0\}$  and $V=Y-\{q_0\}$ where $p_0 \in \mathbb{S}^{4m+3}-\mathbb{S}^{4n+3}$ and $q_0 \in \mathbb{S}^{4l+3}-\mathbb{S}^{4n+3}.$ Then $U$ and $V$ are open subsets of $Y$ and $Y=U \cup V.$ Clearly, $U \simeq \mathbb{S}^{4l+3}, ~ V \simeq \mathbb{S}^{4m+3}$ and $U \cap V \simeq \mathbb{S}^{4n+3}.$ By Mayer-Vietoris cohomology exact sequence we get $H^i(Y;R) \cong R$ for $i=0,4n+4,4m+3~\&~ 4l+3;$ and trivial otherwise, where $R= \mathbb{Z}_p,~p$ a prime or  $\mathbb{Q}.$  It is easy to see that  $Y \sim_{R} \mathbb{S}^{4n+4} \vee \mathbb{S}^{4m+3} \vee \mathbb{S}^{4l+3}.$ This action also induces a free action of  $G=\mathbb{Z}_p,~p$ a prime or  $G=\mathbb{S}^1$ on $Y.$\\
		\indent Similarly, we can  construct a free involution on $Y\sim_{\mathbb{Z}_2} \mathbb{S}^{n} \vee \mathbb{S}^{m} \vee \mathbb{S}^{l}$ and free  actions of $G=\mathbb{Z}_p,~p>2$ a prime and $G=\mathbb{S}^1$ on $Y \sim_{R} \mathbb{S}^{2n+2} \vee \mathbb{S}^{2m+1} \vee \mathbb{S}^{2l+1}$ for suitable choices of $n,m $ and $l.$  Dotzel et al. \cite{f} have constructed an another example of a  free action of $G=\mathbb{Z}_p$, $p$ an odd prime, on a finite CW-complex $Y$ with the mod $p$ cohomology  $\mathbb{S}^n \vee \mathbb{S}^{2n} \vee \mathbb{S}^{3n}.$
		
	\end{example}

	\noindent First, we determine the orbit spaces of free involutions on a finite CW-complex $X \sim_{\mathbb{Z}_2} \mathbb{S}^n \vee \mathbb{S}^m \vee \mathbb{S}^l, 1 \leq n \leq m \leq l.$ 
	
	\begin{theorem}\label{thm 4.4}
		Let $G = \mathbb{Z}_2$ act freely on a finite CW-complex $X \sim_{\mathbb{Z}_2} \mathbb{S}^n \vee \mathbb{S}^m \vee \mathbb{S}^l, 1\leq  n \leq m \leq l.$ Then the graded commutative algebra of the orbit space  is one of the following:
		\begin{enumerate}
			\item $\mathbb{Z}_2 [y,z]/<y^{l+1},z^2+a_0y^{2n}+a_1y^nz,y^{m-n+1}z+a_2y^{m+1}>,$ where deg $y=1,$ deg $z=n$ and $a_0,a_1,a_2 \in \mathbb{Z}_2~\&~ a_0=0$ if $l<2n, a_1=0$ if $m < 2n$ and $a_2=0$ if $l=m,$
			\item $\mathbb{Z}_2 [y,z]/<y^{n+1},z^2+a_0y^nz,y^{l-m+1}z>,$ where deg $y=1,$ deg $z=m$ and $a_0\in \mathbb{Z}_2~\&~ a_0=0$ if $n< m~\mbox{or}~  \{l< 2n~ \&~ n=m\},$ and 
			\item $\mathbb{Z}_2 [y,z]/<y^{m+1},z^2+a_0y^{2n}+a_1y^nz,y^{l-n+1}z>,$ where deg $y=1,$ deg $z=n$ and $a_0,a_1 \in \mathbb{Z}_2~\&~ a_0=0$ if $m<2n~\mbox{and}~ a_1=0$ if $l < 2n.$
		\end{enumerate}
	\end{theorem}
	\begin{proof}
		The $E_2$-term of the Leray-Serre spectral sequence of the Borel fibration $ X \stackrel{i} \hookrightarrow X_G \stackrel{\pi} \rightarrow B_G$ depends upon $\pi_1(B_G)$ acts trivially or nontrivially on $H^*(X).$ So, we consider two cases (1) $\pi_1(B_G)$ acts trivially on $H^*(X),$ and (2) $\pi_1(B_G)$ acts nontrivially on $H^*(X).$\\
		\textbf{Case(1)} When  $\pi_1(B_G)$ acts trivially on $H^*(X).$\\
		In this case the system of local coefficients  of the fibration $X \hookrightarrow X_G \rightarrow B_G$ is simple and hence  $E_2^{k,i} \cong  H^{k}(B_G)\otimes H^i(X).$ Since $G$ acts freely on $X,$ the spectral sequence must be non degenerated. So, we  have $d_r \neq 0$ for some $r \geq 2.$ By the multiplicative property of the spectral sequence we observe that the case $n=m=l$ is not possible and for the other cases the least $r$ for which $d_r \neq 0$ must be $(i)~m-n+1, ~(ii)~l-m+1$ or $~(iii)~l-n+1.$\\ 
		\textbf {Subcase(i)} Let $d_{m-n+1}\neq 0.$  In this case, $n <m \leq l.$ So, we have   $d_{m-n+1}(1 \otimes b)=t^{m-n+1} \otimes a $ when $m<l,$ or $\{d_{m-n+1}(1 \otimes b)= t^{m-n+1}\otimes a~ \mbox{or} ~ d_{m-n+1}(1 \otimes c)= t^{m-n+1}\otimes a~\mbox{or}~d_{m-n+1}(1 \otimes b)=d_{m-n+1}(1 \otimes c)= t^{m-n+1}\otimes a\}$ when $m=l.$  Thus,  $ E_{m-n+2}^{k,i} \cong \mathbb{Z}_2~\forall ~ k\geq 0~ \&~ i=0,l; ~ \mbox{and }~E_{m-n+2}^{k,n} \cong \mathbb{Z}_2, 0 \leq k \leq m-n.$ By Proposition \ref{prop 4.5}, the differential $d_{l+1}$ must be nontrivial. Accordingly, we have  $d_{l+1}(1 \otimes c)= t^{l+1}\otimes 1$ when $m<l$ or \{$d_{l+1}(1 \otimes c)= t^{l+1}\otimes 1 $ or $d_{l+1}(1 \otimes b)= t^{l+1}\otimes 1$ or $d_{l+1}(1 \otimes (b+c))= t^{l+1}\otimes 1$\}  when $m=l.$   Consequently, $E_{l+2}^{*,*}= E^{*,*}_{\infty}$ and hence $E^{k,0}_{\infty}\cong \mathbb{Z}_2, 0 \leq k \leq l;   E^{k,n}_{\infty}\cong \mathbb{Z}_2,0 \leq k \leq m-n;$ and trivial  otherwise. The cohomology groups are given by 	
		\[
		H^j(X_G)=
		\begin{cases}
			\mathbb{Z}_2  & 0 \leq j <  n~\&~m<j \leq l \\
			\mathbb{Z}_2 \oplus \mathbb{Z}_2  & n \leq j \leq m\\
			0 & \text{otherwise}
		\end{cases}
		\]
		Let $y \in E^{1,0}_{\infty}$ and $u \in E_{\infty}^{0,n}$ be the elements corresponding to permanent cocycles  $t \otimes 1 \in E_2^{1,0}$ and $1 \otimes a \in E_2^{0,n},$ respectively. Clearly, $y^{l+1}=u^2=y^{m-n+1}u=0.$ Thus the total complex Tot$E_{\infty}^{*,*}$ is given by $\mathbb{Z}_2[y,u]/<y^{l+1},u^2,y^{m-n+1}u>$ where deg $y=1$ and deg $u=n.$ Let $z \in H^n(X_G)$ corresponds to $u$ such that $i^*(z)=a.$ We get $z^2=a_0y^{2n}+a_1y^nz$ and $y^{m-n+1}z=a_2y^{m+1}$ where $a_0,a_1,a_2 \in \mathbb{Z}_2.$ Note that $a_0=0$ if $l<2n,~ a_1=0$  if $m<2n~ \&~ a_2=0 $ if $m=l.$  Thus the cohomology ring $H^*(X_G)$ is given by $$\mathbb{Z}_2[y,z]/<y^{l+1},z^2+a_0y^{2n}+a_1y^nz,y^{m-n+1}z+a_2y^{m+1}>$$ where deg $y=1$ and deg $z=n.$ By Proposition \ref{tom}, we have $H^*(X_G) \cong H^*(X/G).$  This realizes possibility (1).\\
		\textbf{Subcase(ii)} Let $d_{l-m+1}\neq 0.$ In this case $n \leq m <l.$ We have    $d_{l-m+1}(1 \otimes c)=t^{l-m+1} \otimes b $ when $n<m$ or $\{d_{l-m+1}(1 \otimes c)=t^{l-m+1} \otimes b ~\mbox{or} ~d_{l-m+1}(1 \otimes c)=t^{l-m+1} \otimes a \}$ when $n=m.$ Thus, $\{ E_{l-m+2}^{k,i} \cong \mathbb{Z}_2~\forall ~ k\geq 0~\&~ i=0,n;$ ~and ~ $E_{l-m+2}^{k,m} \cong \mathbb{Z}_2, 0 \leq k \leq l-m\}$ when $n< m$ or $ \{E_{l-m+2}^{k,0} \cong \mathbb{Z}_2~\forall ~ k\geq 0;$ $E_{l-m+2}^{k,m} \cong \mathbb{Z}_2\oplus \mathbb{Z}_2, 0 \leq k \leq l-m;$ and  $E_{l-m+2}^{k,m} \cong \mathbb{Z}_2 ~ \forall~ k > l-m\}$ when $n=m.$ By Proposition \ref{prop 4.5}, $l\leq n+m$ and  the differential $d_{n+1}$ must be nontrivial. Accordingly,  we have  $d_{n+1}(1 \otimes a)= t^{n+1}\otimes 1$ when $n<m$ or $\{d_{n+1}(1 \otimes a)= t^{n+1}\otimes 1 ~\mbox{or}~d_{n+1}(1 \otimes b)= t^{n+1}\otimes 1\}$ when $n=m.$  Consequently, $E_{n+2}^{*,*}= E^{*,*}_{\infty}$ and hence  $E^{k,0}_{\infty}\cong \mathbb{Z}_2, 0 \leq k \leq n; E^{k,m}_{\infty}\cong \mathbb{Z}_2, 0 \leq k \leq l-m;$ and trivial otherwise. The cohomology groups are given by 	
		\[
		H^j(X_G)=
		\begin{cases}
			\mathbb{Z}_2  & 0 \leq j <  n~\&~ m+1\leq j \leq l \\
			\mathbb{Z}_2  & j=n,m ~\& ~n < m\\
			\mathbb{Z}_2 \oplus \mathbb{Z}_2 & j=n ~\&~ n=m\\ 
			0 & \text{otherwise}
		\end{cases}
		\]
		Let  $y \in E^{1,0}_{\infty}$ and $u \in E_{\infty}^{0,n}$ be the elements  corresponding to permanent cocycles $t \otimes 1 \in E_2^{1,0}$ and $\{ 1\otimes b ~\mbox{or}~ 1 \otimes a \in E_2^{0,n} \}$ respectively. Clearly, $y^{n+1} =u^2=y^{l-m+1}u=0.$ Thus the total complex Tot$E_{\infty}^{*,*}$ is given by $\mathbb{Z}_2[y,u]/<y^{n+1},u^2,y^{l-m+1}u>$ where deg $y=1$ and deg $u=m.$ Let $z \in H^n(X_G)$ corresponds to $u$ such that $i^*(z)=b$ or $a.$  Clearly, $z^2=a_0y^nz,a_0\in \mathbb{Z}_2$ where $ a_0=0$ if $n< m~\mbox{or}~  \{l< 2n~ \&~ n=m\}.$  Thus the cohomology ring $H^*(X/G)$ is given by $$\mathbb{Z}_2[y,z]/<y^{n+1},z^2+a_0y^nz,y^{l-m+1}z>$$ where deg $y=1$ and deg $z=m.$ This realizes possibility (2).\\
		\textbf{Subcase(iii)} Let $d_{l-n+1}\neq 0.$ In this case, we consider $n<m<l.$  We have $d_{l-n+1}(1 \otimes c)=t^{l-n+1} \otimes a .$ Thus, $ E_{l-n+2}^{k,i} \cong \mathbb{Z}_2~\forall ~ k\geq 0~ \& ~ i=0,m;$ and $E_{l-n+2}^{k,n} \cong \mathbb{Z}_2, 0 \leq k \leq l-n.$   Clearly, $l\leq m+n$ and $d_{m+1}(1 \otimes b)= t^{m+1}\otimes 1.$ Consequently, $E_{m+2}^{*,*}= E^{*,*}_{\infty}$ and hence $E^{k,0}_{\infty}\cong \mathbb{Z}_2, 0 \leq k \leq m; E^{k,n}_{\infty}\cong \mathbb{Z}_2, 0 \leq k \leq l-n;$ and trivial otherwise. The cohomology groups are given by 	
		\[
		H^k(X_G)=
		\begin{cases}
			\mathbb{Z}_2  & 0 \leq k <  n~\&~m<k \leq l \\
			\mathbb{Z}_2 \oplus \mathbb{Z}_2  & n\leq k \leq m\\
			0 & \text{otherwise}
		\end{cases}
		\]
		Let  $y \in E^{1,0}_{\infty}$ and $u \in E_{\infty}^{0,n}$ be the elements corresponding to permanent cocycles $t \otimes 1 \in E_2^{1,0}$ and $1 \otimes a \in E_2^{0,n},$ respectively. Clearly, $y^{m+1}=u^2=y^{l-n+1}u=0.$  Thus the total complex Tot$E_{\infty}^{*,*}$ is given by $\mathbb{Z}_2[y,u]/<y^{m+1},u^2,y^{l-n+1}u>$ where deg $y=1$ and deg $u=n.$ Let $z \in H^n(X_G)$ corresponds to $u$ such that $i^*(z)=a.$ We get $z^2=a_0y^{2n}+a_1y^nz$ where $a_0,a_1 \in \mathbb{Z}_2$ with $a_0=0$ if $m<2n~\&~ a_1=0 $  if $l<2n.$ Thus the cohomology ring $H^*(X/G)$ is given by $$\mathbb{Z}_2[y,z]/<y^{m+1},z^2+a_0y^{2n}+a_1y^nz,y^{l-n+1}z>$$ where deg $y=1$ and deg $z=n.$ This realizes possibility (3).\\
		\textbf{Case(2)} When $\pi_1(B_G)$ acts nontrivially on $H^*(X).$\\
		In this case  either $n=m\leq l$ or $n<m=l.$ Let $g$ be the generator of $\pi_1(B_G).$\\ If $n=m < l$ then  we have following  three possibilities of nontrivial actions of $\pi_1(B_G)$ on $H^n(X)$: $(i)~ g^*(a)=b, g^*(b)=a$ ~$(ii)~ g^*(a)=a+b, g^*(b)=b$~ $(iii)~ g^*(a)=a, g^*(b)=a+b.$ If $g^*(a)=b$ and $ g^*(b)=a$ then by Proposition \ref{re}, we get $E_2^{0,n}\cong \mathbb{Z}_2$ generated by $ a+b$ and $E_2^{k,n}\cong 0~ \forall ~ k>0.$  As $\pi_1(B_G)$ acts trivially on $H^l(X),$ we get $E^{k,l}_{2}\cong \mathbb{Z}_2$ generated by $t^k \otimes c ~ \forall~ k \geq 0.$ If $d_{n+1}(a+b)= t^{n+1} \otimes 1$ then $0=d_{n+1}((a+b) (t\otimes 1))= t^{n+2} \otimes 1,$ a contradiction. Therefore the differential $d_{n+1}$ is trivial. As $G$ acts freely on $X,$ we must have   the differential $d_{l+1}(1 \otimes c)= t^{l+1}\otimes 1.$ Consequently, $E_{l+2}^{*,*}= E^{*,*}_{\infty}$ and hence $E^{k,0}_{\infty}\cong \mathbb{Z}_2, 0 \leq k \leq l;  E^{0,n}_{\infty} \cong \mathbb{Z}_2;$ and trivial otherwise. Thus,  we have  	
		\[
		H^j(X_G)=
		\begin{cases}
			\mathbb{Z}_2  & 0 \leq j < n, n < j \leq l\\
			\mathbb{Z}_2 \oplus \mathbb{Z}_2 & j=n\\
			0 & \text{otherwise}
		\end{cases}
		\]
		Let $y \in E^{1,0}_{\infty}$ and $u \in E_{\infty}^{0,n}$    be the elements corresponding to permanent cocycle  $t \otimes 1 \in E_2^{1,0}$ and  $a+b \in E_2^{0,n},$  respectively. Clearly, we have  $y^{l+1}= yu=0.$ Thus the total complex Tot$E_{\infty}^{*,*}$ is given by $\mathbb{Z}_2[y,u]/<y^{l+1},u^2,yu>,$ where deg $y=1$ and deg $u=n.$ Let $z \in H^n(X_G)$ corresponds to $u$ such that $i^*(z)=a+b.$ We get $z^2=a_0 y^{2n}~\&~ yz=a_2y^{n+1}$ where $a_0,a_2 \in \mathbb{Z}_2~\&~$  $a_0=0$ if $l<2n.$ Thus the cohomology ring $H^*(X_G)$ is given by $$\mathbb{Z}_2[y,z]/<y^{l+1},z^2+a_0 y^{2n},yz+a_2y^{n+1}>$$ where deg $y=1$ and deg $z=n.$ This realizes possibility (1) when $n<m=l$. For other nontrivial actions we get the same possibility.\\
		For $n<m=l,$ we have similar nontrivial actions as above  when $n=m<l.$ For any of these nontrivial actions it is clear that the  cohomology ring of the orbit space is given by $$\mathbb{Z}_2[y,z]/<y^{n+1},z^2,yz>$$ where deg $y=1$ and deg $z=m.$ This realizes possibility (2) when $n<m=l.$\\
		If $n=m=l$ then it is easy to see that $\pi_1(B_G)$ can act nontrivially on $H^n(X)$ by eighteen ways and $g^*$ must fixes one or two generator(s) of $H^n(X).$ Now, consider a nontrivial action  defined as $g^*(a)=a,g^*(b)=c$ and $g^*(c)=b.$  By Proposition \ref{re}, we have  $E_2^{0,n}\cong \mathbb{Z}_2 \oplus \mathbb{Z}_2,$ generated by $a~\&~ b+c$ and $E_2^{k,n}\cong \mathbb{Z}_2~ \forall ~ k>0,$ generated by $(t^k \otimes 1)a.$ %Let $\alpha $ and $\beta$ be the generator of $E^{0,n}_2$ represented by $b$ and $a+c$ respectively.
		As $G$ acts freely on $X,$  we must have the differential $d_{n+1}: E_{n+1}^{0,n} \rightarrow E_{n+1}^{0,n}$ is nontrivial. If    $d_{n+1}(a)=0$ then two lines survives to infinity, a contradiction. So, $d_{n+1}(a)\neq 0.$  Thus, $E_{n+2}^{*,*}=E_{\infty}^{*,*}$ and hence $E^{k,0}_{\infty}\cong \mathbb{Z}_2, 0 \leq k \leq n; E^{0,n}_{\infty}\cong \mathbb{Z}_2;$ and trivial otherwise. Note that $E_{\infty}^{0,n}$ is generated by $b+c$ or $a+b+c$ accordingly  $d_{n+1}(b+c)$ is zero or nonzero. Consequently, we have 
		\[
		H^j(X_G)=
		\begin{cases}
			\mathbb{Z}_2  & 0 \leq j < n \\
			\mathbb{Z}_2 \oplus \mathbb{Z}_2 & j=n\\
			0 & \text{otherwise}
		\end{cases}
		\]
		Let $y \in E^{1,0}_{\infty}$ and  $u \in E_{\infty}^{0,n}$ be elements  corresponding to permanent cocycles $t \otimes 1 \in E_2^{1,0}$ and $b+c~ \mbox{or}~a+b+c \in E_2^{0,n},$  respectively. We have $u^2=yu=0.$ Thus the total complex Tot$E_{\infty}^{*,*}$ is given by $\mathbb{Z}_2[y,u]/<y^{n+1},u^2,yu>,$ where deg $y=1$ and deg $u=n.$ Let $z \in H^n(X_G)$ corresponds to $u$ such that $i^*(z)=b ~\mbox{or} ~a+b+c.$ Clearly,  the cohomology ring $H^*(X_G)$ is the same as in possibility (2). For other possibilities of nontrivial actions of $\pi_1(B_G)$ on $H^n(X),$ we will get same  cohomological algebra of the orbit space.
	\end{proof}
	\noindent	Note that if $n=m<l$ and  $a_0=a_1=a_2=0$ in possibility (1) of above Theorem then $X/G \sim_{\mathbb{Z}_2} \mathbb{RP}^l \vee \mathbb{S}^n,$ and if $n\leq m=l$ and $a_2=0$ in  possibility (2) then $X/G \sim_{\mathbb{Z}_2} \mathbb{RP}^n \vee \mathbb{S}^m.$\\
	
	\noindent Next, we determine the orbit spaces of free  $G=\mathbb{Z}_p,~ p>2$ a prime, actions on a finite CW-complex $X \sim_{\mathbb{Z}_p} \mathbb{S}^n \vee \mathbb{S}^m \vee \mathbb{S}^l, 1 \leq n \leq m \leq l.$ 
	\begin{lemma}\label{lemma}
		Let $G=\mathbb{Z}_p, ~p>2$ a prime, act freely on a finite  CW-complex  $X \sim_{\mathbb{Z}_p} \mathbb{S}^n \vee \mathbb{S}^m \vee \mathbb{S}^l, 1\leq n \leq m \leq l.$ Then $\pi_1(B_G)$  acts trivially on $H^*(X).$
	\end{lemma}
	\begin{proof}
		Clearly, $\chi(F)=0.$ By the Floyd's formula,  $\chi(X)=0.$ Thus, one of  $n,m$ or $l$ must be even and the remaining two are odd.  Suppose that $\pi_1(B_G)= \mathbb{Z}_p$ acts nontrivailly on $H^*(X).$ So,  we must have  either $n=m<l$ or $n<m=l.$ By Proposition \ref{re}, we get $E_2^{k,i}\cong \mathbb{Z}_p~ \forall ~k\geq 0 ~\&~ i=0,n ~\mbox{and}~ l.$ It is clear that atleast one line in the spectral sequence of the fibration $X \hookrightarrow X_G \rightarrow B_G$ survive to infinity, which contradicts  Proposition \ref{prop 4.5}. Hence, our claim. 
	\end{proof}
	\begin{theorem}\label{zp action}
		Let $G = \mathbb{Z}_p,~ p>2$ a prime, act freely on a finite CW-complex $X \sim_{\mathbb{Z}_p} \mathbb{S}^n \vee \mathbb{S}^m \vee \mathbb{S}^l, 1\leq  n \leq m \leq l.$ Then the graded commutative algebra of the orbit space  is one of the following:
		\begin{enumerate}
			\item $\mathbb{Z}_p[x,y,z]/<x^2,y^{\frac{l+1}{2}},z^2-a_0y^n-a_1x^{b_1}y^{\frac{n-{b_1}}{2}}z,y^{\frac{m-n+1}{2}}z-a_2x^{1-{b_2}}y^{\frac{m+{b_2}}{2}}>,$ where deg $x=1, y= \beta (x)~\&~$  deg $z=n,$ either $n$ or $m$ is even, and $b_1,b_2 \in \{0,1\},a_i \in \mathbb{Z}_p, 0\leq i \leq 2,a_0=0$  if $l<2n,a_1=0$ if $m<2n~\&$ $a_2=0$ if $m=l,$  
			\item $\mathbb{Z}_p[x,y,z]/<x^2,y^{\frac{n+1}{2}},z^2-a_0xy^{\frac{n-1}{2}} z,y^{\frac{l-m+1}{2}}z>,$ where deg $x=1, y= \beta (x) ~\&~$   deg $z=m,$ either $m$ or $l$ is even,  and $ a_0 \in \mathbb{Z}_p, a_0=0$ if $n<m$ or $\{l<2n ~\&~ n=m\},$ and 
			\item $\mathbb{Z}_p[x,y,z]/<x^2,y^{\frac{m+1}{2}},z^2-a_0y^n-a_1x^{b_1}y^{\frac{n-b_1}{2}}z,y^{\frac{l-n+1}{2}}z>,$ where deg $x=1, y= \beta (x)~\&~$  deg $z=n,$ either $n$ or $l$ is even,  and $b_1 \in \{0,1\},a_i \in \mathbb{Z}_p, 0\leq i \leq 1,$ $ a_0 =0$ if $m<2n~\&$ $a_1=0$ if $l<2n,$
		\end{enumerate}
		where  $\beta$ is the mod $p$ bockstein cohomology operation associated to the coefficient sequence $0 \rightarrow \mathbb{Z}_p \rightarrow \mathbb{Z}_{{p}^2}\rightarrow \mathbb{Z}_p \rightarrow 0.$
		
	\end{theorem}
	\begin{proof}
		By Lemma \ref{lemma},  the $E_2$-term of the Leray-Serre spectral sequence associated to the Borel
		fibration $ X \stackrel{i} \hookrightarrow X_G \stackrel{\pi} \rightarrow B_G$ is given by $E^{k,i}_2= H^{k}(B_G) \otimes H^i(X).$  Since $G$ acts freely on $X,$ $E_2 \neq E_{\infty}$ and one of  $n,m$ or $l$ must be even and the remaining two are odd. It is clear that the least $r$ for which $d_r \neq 0$ must be $~(i)~m-n+1,~(ii)~l-m+1$ or $~(iii)~l-n+1.$ \\
		\textbf{Case(i)} Let $d_{m-n+1}\neq 0.$ If  $l$ is even then  $H^i(X/G) \neq 0$   for odd $  i >l,$  which contradicts   Proposition \ref{prop 4.5}. So, we consider  when $n$ or $m$ is even.\\
		If $n$ is even then  $n<m\leq l.$ We have  $d_{m-n+1}(1 \otimes b)= t^{\frac{m-n+1}{2}}\otimes a$ when $m<l,$ or $\{d_{m-n+1}(1\otimes b)=t^{\frac{m-n+1}{2}}\otimes a ~ \mbox{or}~d_{m-n+1}(1\otimes c)=t^{\frac{m-n+1}{2}}\otimes a ~ \mbox{or } d_{m-n+1}(1\otimes b)=d_{m-n+1}(1\otimes c)=t^{\frac{m-n+1}{2}} \otimes  a\}$ when $m=l.$ Thus, $ E_{m-n+2}^{k,i} \cong \mathbb{Z}_p~\forall ~ k\geq 0~ \&~ i=0,l; ~ \mbox{and }~E_{m-n+2}^{k,n} \cong \mathbb{Z}_p, 0 \leq k \leq m-n.$  Since $G$ act freely on $X,$  $d_{l+1} \neq 0.$ So, we have  $d_{l+1}(1 \otimes c )=t^{\frac{l+1}{2}}\otimes 1$ when $m<l$ or $\{d_{l+1}(1 \otimes c )=t^{\frac{l+1}{2}}\otimes 1~ \mbox{or }~ d_{l+1}(1 \otimes b )=t^{\frac{l+1}{2}}\otimes 1~\mbox{or}~d_{l+1}(1 \otimes (b+c) )=t^{\frac{l+1}{2}}\otimes 1\}$ when $m=l.$  Consequently, $E_{l+2}^{*,*}=E_{\infty}^{*,*}$ and hence $E_{\infty}^{k,0}= \mathbb{Z}_p, 0 \leq k \leq l;$ $ E_{\infty}^{k,n}= \mathbb{Z}_p, 0 \leq k \leq m-n;$ and trivial otherwise. The cohomology groups are given by 	
		\[
		H^j(X_G)=
		\begin{cases}
			\mathbb{Z}_p  & 0 \leq j <  n~\&~m<j \leq l \\
			\mathbb{Z}_p \oplus \mathbb{Z}_p  & n \leq j \leq m\\
			0 & \text{otherwise}
		\end{cases}
		\]
		Let $x \in E^{1,0}_{\infty},~ y \in E_{\infty}^{2,0}$ and   $u \in E_{\infty}^{0,n}$ be elements  corresponding to permanent cocycles $s \otimes 1 \in E_2^{1,0},~ t \otimes 1 \in E_{2}^{2,0}$ and $1 \otimes a \in E_2^{0,n},$   respectively. Clearly, $y^{\frac{l+1}{2}}=u^2=y^{\frac{m-n+1}{2}}u=0.$  Thus the total complex Tot$E_{\infty}^{*,*}$ is given by $\mathbb{Z}_p[x,y,u]/<x^2,y^{\frac{l+1}{2}},u^2,y^{\frac{m-n+1}{2}}u>$ where deg $x=1$ deg $y= 2$ and deg $u=n.$ Let $z \in H^n(X_G)$ corresponds to $u$ such that $i^*(z)=a.$ We get   $z^2=a_0y^n+a_1y^{\frac{n}{2}}z,$ where $a_0,a_1 \in \mathbb{Z}_p.$ Note that $a_0=0$ if $l<2n~\&~ a_1=0 $  if $m<2n$ and $y^{\frac{m-n+1}{2}}z=a_2y^{\frac{m+1}{2}}, a_2 \in \mathbb{Z}_p, a_2=0$ if $m=l.$ Thus the cohomology ring $H^*(X_G)$ is given by $$\mathbb{Z}_p[x,y,z]/<x^2,y^{\frac{l+1}{2}},z^2-a_0y^n-a_1y^{\frac{n}{2}}z,y^{\frac{m-n+{1}}{2}}z-a_2y^{\frac{m+1}{2}}>$$ where deg $x=1,$ deg$ y= 2,$   deg $z=n$ and $ y= \beta (x)$ where $\beta$ is bockstein homomorphism. This realizes possibility (1) by taking $b_1=0,b_2=1.$\\ 
		If $m$ is even then $n<m<l.$ The cohomology groups $H^j(X_G)$ and the total complex Tot$E_{\infty}^{*,*}$ are same as above. Clearly, $z^2=a_0y^{n}+a_1xy^{\frac{n-1}{2}}z,$ where $a_0,a_1 \in \mathbb{Z}_p$ and $a_0=0$ if $l<2n~\&~ a_1=0 $  if $m<2n$ and $y^{\frac{m-n+1}{2}}z=a_2xy^{\frac{m}{2}}, a_2 \in \mathbb{Z}_p.$ Thus the cohomology ring $H^*(X_G)$ is given by $$\mathbb{Z}_p[x,y,z]/<x^2,y^{\frac{l+1}{2}},z^2-a_0y^n-a_1xy^{\frac{n-1}{2}}z,y^{\frac{m-n+1}{2}}z-a_2xy^{\frac{m}{2}}>$$ where deg $x=1, y= \beta (x)$ and deg $z=n.$ This realizes possibility (1) by taking $b_1=1,b_2=0.$ \\ 
		\textbf{Case(ii)} Let $d_{l-m+1}\neq 0.$ Clearly, $n$ cannot be even and hence,  either $m$  or $l$ is even.\\
		If $l$ is even then $n\leq m <l.$ We have $d_{l-m+1}(1 \otimes c)= t^{\frac{l-m+1}{2}}\otimes b$ when $n<m$ or $\{d_{l-m+1}(1 \otimes c)= t^{\frac{l-m+1}{2}}\otimes b~\mbox{or}~d_{l-m+1}(1 \otimes c)= t^{\frac{l-m+1}{2}}\otimes a\}$ when $n=m.$ Thus, $\{ E_{l-m+2}^{k,i} \cong \mathbb{Z}_2~\forall ~ k\geq 0~\&~ i=0,n;$ ~and ~ $E_{l-m+2}^{k,m} \cong \mathbb{Z}_p, 0 \leq k \leq l-m\}$ when $n< m,$ or $\{ E_{l-m+2}^{k,0} \cong \mathbb{Z}_p~\forall ~ k\geq 0;$ $E_{l-m+2}^{k,m} \cong \mathbb{Z}_p\oplus \mathbb{Z}_p, 0 \leq k \leq l-m;$ and  $E_{l-m+2}^{k,m} \cong \mathbb{Z}_2 ~ \forall~ k > l-m\}$ when $n=m.$ Since $G$ act freely on $X,$ $l\leq m+n$ and $d_{n+1}\neq 0.$ So, we have  $d_{n+1}(1 \otimes a )=t^{\frac{n+1}{2}}\otimes 1$ when $n<m$ or $\{d_{n+1}(1 \otimes a )=t^{\frac{n+1}{2}}\otimes 1~ \mbox{or}~d_{n+1}(1 \otimes b )=t^{\frac{n+1}{2}}\otimes 1\}$ when $n=m.$ Consequently, $E_{n+2}^{*,*}=E_{\infty}^{*,*}$ and hence $E_{\infty}^{k,0}= \mathbb{Z}_p, 0 \leq k \leq n;$ $ E_{\infty}^{k,m}= \mathbb{Z}_p, 0 \leq k \leq l-m;$ and trivial otherwise. We have
		\[
		H^j(X_G)=
		\begin{cases}
			\mathbb{Z}_p  & 0 \leq j <  n~\&~m+1\leq j \leq l \\
			\mathbb{Z}_p  & j=n,m ~\& ~n < m\\
			\mathbb{Z}_p \oplus \mathbb{Z}_p & j=n ~\&~ n=m\\ 
			0 & \text{otherwise}
		\end{cases}
		\]
		Let  $x \in E^{1,0}_{\infty} ~\&~ y \in E_{\infty}^{2,0}$ and $u \in E_{\infty}^{0,m}$ be elements corresponding to permanent cocycles  $s \otimes 1 \in E_2^{1,0},~ t \otimes 1 \in E_{2}^{2,0}$ and $\{(1 \otimes b)~ \mbox{or}~(1 \otimes a)\} \in E_2^{0,m},$ respectively. Clearly, $x^2=y^{\frac{n+1}{2}}=u^2=y^{\frac{l-m+1}{2}}u=0.$  Thus the total complex Tot$E_{\infty}^{*,*}$ is given by $\mathbb{Z}_p[x,y,u]/<x^2,y^{\frac{n+1}{2}},u^2,y^{\frac{l-m+1}{2}}u>$ where deg $x=1,y= \beta (x)$ and deg $u=m.$ Let $z \in H^m(X_G)$ corresponds to $u$ such that $i^*(z)=b$ or $a.$ Note that $z^2=a_0xy^{\frac{n-1}{2}} z,~ a_0\in \mathbb{Z}_p$ where $ a_0=0$ if $n< m$ or $\{n=m~\&~l<2n \}.$ Thus the cohomology ring $H^*(X_G)$ is given by $$\mathbb{Z}_p[x,y,z]/<x^2,y^{\frac{n+1}{2}},z^2-a_0xy^{\frac{n-1}{2}} z,y^{\frac{l-m+1}{2}}z>$$ where deg $x=1, y= \beta (x)$ and deg $z=m.$ This realizes possibility (2).\\
		If $m$ is even then $n<m<l.$ The cohomology groups $H^j(X_G)$ and the total complex Tot$E_{\infty}^{*,*}$ are same as above. Clearly the cohomology ring $H^*(X_G)$ is given by $$\mathbb{Z}_p[x,y,z]/<x^2,y^{\frac{n+1}{2}},z^2,y^{\frac{l-m+1}{2}}z>$$ where deg $x=1, y= \beta (x)$ and deg $z=m.$ This realizes possibility (2) . \\
		\textbf{Case(iii)} Let $d_{l-n+1}\neq 0.$ In this case, $m$ is not even and hence,  either $n$  or $l$ is  even.\\ If  $n$ is even then $n<m\leq l.$ We have $d_{l-n+1}(1 \otimes c)= t^{\frac{l-n+1}{2}}\otimes a$ when $m<l$ or  $\{d_{l-n+1}(1 \otimes c)= t^{\frac{l-n+1}{2}}\otimes a~\mbox{or}~d_{l-n+1}(1 \otimes b)= t^{\frac{l-n+1}{2}}\otimes a\}$ when $m=l.$ Thus, $ E_{l-n+2}^{k,i} \cong \mathbb{Z}_p~\forall ~ k\geq 0~ \&~ i=0,m; ~ \mbox{and }~E_{l-n+2}^{k,n} \cong \mathbb{Z}_p, 0 \leq k \leq l-n.$ As $G$ act freely on $X,$ $l\leq m+n$ and $d_{m+1} \neq 0.$ So, we have   $d_{m+1}(1 \otimes b )=t^{\frac{m+1}{2}}\otimes 1$ when $m<l$ or  \{$d_{m+1}(1 \otimes b )=t^{\frac{m+1}{2}}\otimes 1$ or $d_{m+1}(1 \otimes c )=t^{\frac{m+1}{2}}\otimes 1$\} when $m=l.$ Therefore, $E_{m+2}^{*,*}=E_{\infty}^{*,*}$ and hence $E_{\infty}^{k,0}= \mathbb{Z}_p, 0 \leq k \leq m;$ $ E_{\infty}^{k,n}= \mathbb{Z}_p, 0 \leq k \leq l-n;$ and trivial otherwise. Thus, we have 	
		\[
		H^j(X_G)=
		\begin{cases}
			\mathbb{Z}_p  & 0 \leq j <  n~\&~m<j \leq l \\
			\mathbb{Z}_p \oplus \mathbb{Z}_p  & n \leq j \leq m\\
			0 & \text{otherwise}
		\end{cases}
		\]
		Let  $x \in E^{1,0}_{\infty},~ y \in E_{\infty}^{2,0}$  and $u \in E_{\infty}^{0,n}$ be  elements corresponding to permanent cocycle $s \otimes 1 \in E_2^{1,0},~ t \otimes 1 \in E_{2}^{2,0}$  and $1 \otimes a \in E_2^{0,n},$ respectively. Clearly, $x^2=y^{\frac{m+1}{2}}=u^2=y^{\frac{l-n+1}{2}}u=0.$   Thus the total complex Tot$E_{\infty}^{*,*}$ is given by $\mathbb{Z}_p[x,y,u]/<x^2,y^{\frac{m+1}{2}},u^2,y^{\frac{l-n+1}{2}}u>$ where deg $x=1,y= \beta (x)$ and deg $u=n.$ Let $z \in H^n(X_G)$ corresponds to $u$ such that $i^*(z)=a.$ We get $z^2=a_0y^n+a_1y^{\frac{n}{2}}z,$  $a_0,a_1 \in \mathbb{Z}_p$ where $a_0=0$ if $m<2n~\&~ a_1=0 $  if $l<2n.$ Thus the cohomology ring $H^*(X_G)$ is given by $$\mathbb{Z}_p[x,y,z]/<x^2,y^{\frac{m+1}{2}},z^2-a_0y^n-a_1y^{\frac{n}{2}}z,y^{\frac{l-n+1}{2}}z>$$ where deg $x=1, y= \beta (x)$ and deg $z=n.$ This realizes case (3) by taking $b_1=0.$\\ 
		If $l$ is even then we consider $n< m<l.$ The cohomology groups $H^j(X_G)$ and the total complex Tot$E_{\infty}^{*,*}$ are same as above. Clearly, $z^2=a_0y^n+a_1xy^{\frac{n-1}{2}}z,$  $a_0,a_1 \in \mathbb{Z}_p$ where $a_0=0$ if $m<2n~\&~ a_1=0 $  if $l<2n.$ Thus the cohomology ring $H^*(X_G)$ is given by $$\mathbb{Z}_p[x,y,z]/<x^2,y^{\frac{m+1}{2}},z^2-a_0y^n-a_1xy^{\frac{n-1}{2}}z,y^{\frac{l-n+1}{2}}z>$$ where deg $x=1, y= \beta (x)$ and deg $z=n.$ This realizes possibility (3) by taking $b_1=1.$	
	\end{proof}
	\noindent Now, we determine the orbit spaces of free actions of  $G=\mathbb{S}^1$ on a finite CW-complex $X \sim_{\mathbb{Q}} \mathbb{S}^n \vee \mathbb{S}^m \vee \mathbb{S}^l, 1 \leq n \leq m \leq l.$

	\begin{theorem}\label{thm 4.7}
		Let $G=\mathbb{S}^1$ act freely on a finite  CW-complex  $X \sim_{\mathbb{Q}} \mathbb{S}^n \vee \mathbb{S}^m \vee \mathbb{S}^l, 1\leq n \leq m \leq l.$ Then the graded commutative algebra of the orbit space  is one of the following:
		\begin{enumerate}
			\item $\mathbb{Q} [y,z]/<y^{\frac{l+1}{2}},z^2-a_0y^{n}-a_1y^{\frac{n}{2}}z,y^{\frac{m-n+1}{2}}z-a_2y^{\frac{m+1}{2}}>,$ where deg $y=2,$ deg $z=n,$ $m-n ~\&~ l$ are odd, and $a_0,a_1,a_2 \in \mathbb{Q},a_0=0$ if $l<2n,a_1=0$ if $n$ is odd or $m<2n~\&$  $a_2=0$ if $m$ is even or $m=l,$
			\item $\mathbb{Q} [y,z]/<y^{\frac{n+1}{2}},z^2,y^{\frac{l-m+1}{2}}z>,$  where deg $y=2,$  deg $z=m,$ $l-m~\&~ n$ are odd, and 
			\item $\mathbb{Q}[y,z]/<y^{\frac{m+1}{2}},z^2-a_0y^{n}-a_1y^{\frac{n}{2}}z,y^{\frac{l-n+1}{2}}z>,$ where deg $y=2,$ deg $z=n,$ $l-n ~\&~ m$ are odd, and  $a_0,a_1 \in \mathbb{Q},a_0=0$ if $m<2n~\&$ $a_1=0$ if $n$ is odd or $l<2n.$ 
			
		\end{enumerate}
	\end{theorem}
	\begin{proof}
		As $\pi_1(B_G)=1$ and $H^*(B_G)$ is torsion free, the $E_2$-term of the Leray-Serre spectral sequence associated to the Borel
		fibration $ X \stackrel{i} \hookrightarrow X_G \stackrel{\pi} \rightarrow B_G$ is given by $E^{k,i}_2= H^{k}(B_G) \otimes H^i(X).$ Since $G$ acts freely on $X,$ we get $E_2^{*,*} \neq E_{\infty}^{*,*}.$ Clearly, the least $r$ for which $d_r\neq 0$ is one of the following: (i) $m-n+1,~$ (ii) $l-m+1$ or~ (iii)~ $l-n+1.$\\
		\textbf{Case(i)} Let $d_{m-n+1} \neq 0.$ Then  $m-n$ must be odd and $d_{m-n+1}(1 \otimes b)= t^{\frac{m-n+1}{2}} \otimes a $ when $m<l$ or \{$d_{m-n+1}(1 \otimes b)= t^{\frac{m-n+1}{2}} \otimes a $ or $d_{m-n+1}(1 \otimes c)= t^{\frac{m-n+1}{2}} \otimes a $ or $d_{m-n+1}(1 \otimes b)= d_{m-n+1}(1 \otimes c)= t^{\frac{m-n+1}{2}} \otimes a$\} when $m=l.$ By Proposition \ref{prop 4.5}, $l$ is odd and $d_{l+1}\neq 0.$  So, we have   $d_{l+1}(1 \otimes c)=t^{\frac{l+1}{2}}\otimes 1$ when $m<l$ or  \{$d_{l+1}(1 \otimes c)=t^{\frac{l+1}{2}}\otimes 1$ or $d_{l+1}(1 \otimes b)=t^{\frac{l+1}{2}}\otimes 1$ or $d_{l+1}(1 \otimes (b+c))=t^{\frac{l+1}{2}}\otimes 1$  \} when $m=l.$  Consequently, $E_{l+2}^{*,*}= E^{*,*}_{\infty}$ and hence $E^{k,0}_{\infty}= \mathbb{Q}, 0 \leq k ~\mbox{even} < l; E^{k,n}_{\infty}= \mathbb{Q},0 \leq k~\mbox{even} < m-n;$ and trivial otherwise.  
		For $n$ even, we get	
		\[
		H^j(X_G)=
		\begin{cases}
			\mathbb{Q}  & 0 \leq j ~\mbox{even}<  n~\&~m<j~\mbox{even} < l \\
			\mathbb{Q} \oplus \mathbb{Q}  & n\leq j ~\mbox{even} < m\\
			0 & \text{otherwise}
		\end{cases}
		\] 
		For $n$ odd, we get 
		\[
		H^j(X_G)=
		\begin{cases}
			\mathbb{Q}  & 0 \leq j ~\mbox{even} <  n,n\leq j  <m ~\&~m\leq j~\mbox{even} < l \\
			
			0 & \text{otherwise}
		\end{cases}
		\] 
		Let $y \in E_{\infty}^{2,0}$  and $u \in E_{\infty}^{0,n} $ be elements corresponding to the permanent cocyles $t\otimes 1 $ and $1\otimes a,$ respectively.  Clearly, $y^{\frac{l+1}{2}}=u^2=y^{\frac{m-n+1}{2}}u=0.$ Then the total complex is given by Tot$E_{\infty}^{*,*}=\mathbb{Q}[y,u]/<y^{\frac{l+1}{2}},u^2,y^{\frac{m-n+1}{2}}u>$ where deg $y=2$ and deg $u=n.$ We get an element $z \in H^n(X_G)$ corresponding to $u$ such that $i^*(z)=a.$ Clearly, $z^2=a_0y^{n}+a_1y^{\frac{n}{2}}z ~\&~ y^{\frac{m-n+1}{2}}z=a_2y^{\frac{m+1}{2}},$ $a_0,a_1~\&~a_2 \in \mathbb{Q},$ where $a_0=0$ if $ l <2n,$ $a_1=0$ if $n$ is odd or $m<2n$ and $a_2=0$ if $m$ is even or $m=l.$ Thus the cohomology ring $H^*(X/G)$ is given by $$\mathbb{Q}[y,z]/<y^{\frac{l+1}{2}},z^2-a_0y^{n}-a_1y^{\frac{n}{2}}z,y^{\frac{m-n+1}{2}}z-a_2y^{\frac{m+1}{2}}>$$ where deg $y=2$ and deg $z=n.$   This realizes possibility (1).\\
		\textbf{Case(ii)} Let $d_{l-m+1}\neq 0.$ Then   $l-m $ must be an odd and $d_{l-m+1}(1 \otimes c)= t^{\frac{l-m+1}{2}} \otimes b$ when $n<m$ or \{$d_{l-m+1}(1 \otimes c)= t^{\frac{l-m+1}{2}} \otimes b$  or $d_{l-m+1}(1 \otimes c)= t^{\frac{l-m+1}{2}} \otimes a$\} when $n=m.$  As $G$ acts freely on $X,$ $l\leq n+m$ and  $d_{n+1} \neq 0.$  So, $n$ must be odd, and $d_{n+1}(1 \otimes a)=t^{\frac{n+1}{2}} \otimes 1$ when $n<m$ or \{$d_{n+1}(1 \otimes a)=t^{\frac{n+1}{2}} \otimes 1$ or $d_{n+1}(1 \otimes b)=t^{\frac{n+1}{2}} \otimes 1$\} when $n=m.$ Consequently, $E_{n+2}^{*,*}= E_{\infty}^{*,*}$ and  hence $E_{\infty}^{k,0}= \mathbb{Q}$ for $0 \leq k ~\mbox{even}<n;$  $E_{\infty}^{k,m}=\mathbb{Q}$ for $ 0 \leq k ~\mbox{even}< l-m;$ and trivial otherwise.
		For $m$ odd, we get 
		\[
		H^j(X_G)=
		\begin{cases}
			\mathbb{Q}  & 0 \leq j ~\mbox{even}<  n~\&~m\leq j ~\mbox{odd} <l  \\
			
			0 & \text{otherwise}
		\end{cases}
		\] 
		For $m$ even, we get
		\[
		H^j(X_G)=
		\begin{cases}
			\mathbb{Q}  & 0 \leq j ~\mbox{even}<  n~\&~m\leq j ~\mbox{even} <l\\
			
			0 & \text{otherwise}
		\end{cases}
		\] 
		Let  $y \in E_{\infty}^{2,0}$ and $u \in E_{\infty}^{0,m} $  be elements corresponding to the permanent cocyles $t\otimes 1$ and  \{$1\otimes a$ or $1 \otimes b$\},   respectively. Clearly, $y^{\frac{n+1}{2}}=u^2=y^{\frac{l-m+1}{2}}u=0.$ Then the total complex is given by Tot$E_{\infty}^{*,*}=\mathbb{Q}[y,u]/<y^{\frac{n+1}{2}},u^2,y^{\frac{l-m+1}{2}}u>$ where deg $y=2$ and deg $u=m.$ We get an element $z \in H^m(X_G)$ corresponding to $u$ such that $i^*(z)=b$ or $a$. Thus the cohomology ring $H^*(X/G)$ is given by $$\mathbb{Q}[y,z]/<y^{\frac{n+1}{2}},z^2,y^{\frac{l-m+1}{2}}z>$$ where deg $y=2$ and deg $z=m.$  This realizes possibility (2).\\
		$\textbf{Case(iii)}$ Let $d_{l-n+1}\neq 0.$ Then $l-n $ must be  odd and we consider $n<m<l.$ We have  $d_{l-m+1}(1 \otimes c)= t^{\frac{l-n+1}{2}} \otimes a.$ As $G$ acts freely on $X,$ $l\leq n+m$ and $d_{m+1}\neq 0.$ So, $m$ must be odd, and $d_{m+1}(1 \otimes b)=t^{\frac{m+1}{2}} \otimes 1.$ Consequently, $E_{m+2}^{*,*}= E_{\infty}^{*,*}$ and hence $E_{\infty}^{k,0}= \mathbb{Q}$ for $0 \leq k ~\mbox{even}<m;$ $E_{\infty}^{k,n}=\mathbb{Q}$ for $ 0 \leq k ~\mbox{even}< l-n;$ and trivial otherwise.  For $n$ odd, we get
		\[
		H^j(X_G)=
		\begin{cases}
			\mathbb{Q}  & 0 \leq j ~\mbox{even}<  n,n\leq j < m ~\&~ m<j ~\mbox{odd}<l \\
			
			0 & \text{otherwise}
		\end{cases}
		\] 
		For $n$ even, we get 
		\[
		H^j(X_G)=
		\begin{cases}
			\mathbb{Q}  & 0 \leq j ~\mbox{even}<  n~\&~m < j ~\mbox{even} <l\\
			\mathbb{Q} \oplus \mathbb{Q}  & n\leq j ~\mbox{even} < m\\
			0 & \text{otherwise}
		\end{cases}
		\] 
		Let $y \in E_{\infty}^{2,0}$ and $u\in E_{\infty}^{0,n} $ be elements corresponding to the permanent cocyles  $t\otimes 1$  and $1\otimes a,$ respectively. Clearly $y^{\frac{m+1}{2}}=u^2=y^{\frac{l-n+1}{2}}u=0.$  Then the total complex is given by Tot$E_{\infty}^{*,*}=\mathbb{Q}[y,u]/<y^{\frac{m+1}{2}},u^2,y^{\frac{l-n+1}{2}}u>$ where deg $y=2$ and deg $u=n.$ We get an element $z \in H^m(X_G)$ corresponding to $u$ such that $i^*(z)=a.$ Clearly, $z^2=a_0y^n+a_1y^{\frac{n}{2}}z,$ $a_0,a_1~\&~a_2 \in \mathbb{Q}$ where $a_0=0$ if $m<2n,$  $a_1=0$ if $n$ is odd  or $l<2n.$ Thus the cohomology ring $H^*(X/G)$ is given by $$\mathbb{Q}[y,z]/<y^{\frac{m+1}{2}},z^2-a_0y^n-a_1y^{\frac{n}{2}}z,y^{\frac{l-n+1}{2}}z>$$ where deg $y=2$ and deg $z=n.$  This realizes possibility (3).
	\end{proof}
	
	\noindent Finally, we determine the orbit spaces of free actions of   $G=\mathbb{S}^3$ on a finite CW-complex $X \sim_{\mathbb{Q}} \mathbb{S}^n \vee \mathbb{S}^m \vee \mathbb{S}^l, 1 \leq n \leq m \leq l.$ Here, we assume that the associated sphere bundle $G \hookrightarrow X \rightarrow X/G$ is orientable. 
	\begin{theorem} \label{thm 4.8}
		Let $G=\mathbb{S}^3$ act freely on a finite  CW-complex  $X \sim_{\mathbb{Q}} \mathbb{S}^n \vee \mathbb{S}^m \vee \mathbb{S}^l, 1\leq n \leq m \leq l.$   Then the graded commutative algebra of the orbit space  is one of the following:
		\begin{enumerate}
			\item $\mathbb{Q} [y,z]/<y^{\frac{l+1}{4}},z^2-a_0y^{\frac{n}{2}}-a_1y^{\frac{n}{4}}z,y^{\frac{m-n+1}{4}}z-a_2y^{\frac{m+1}{4}}>,$ where  deg $y=4,$ deg $z=n,$  $m-n \equiv 3$(mod $4$), $ l\equiv 3$(mod $4$), and  $a_0,a_1,a_2 \in \mathbb{Q},$ $a_0=0$ if $n$ is odd or $l<2n,$ $a_1=0$ if $n\not\equiv 0$(mod $4$) or $m<2n~\&~a_2=0$ if $m+1 \not\equiv 0$(mod $4$) or $m=l,$ 
			\item $\mathbb{Q} [y,z]/<y^{\frac{n+1}{4}},z^2,y^{\frac{l-m+1}{4}}z>$  where deg $y=4,$  deg $=m$ and $l-m \equiv 3$(mod $4$), $n\equiv 3$(mod $4$), and
			\item $\mathbb{Q}[y,z]/<y^{\frac{m+1}{4}},z^2-a_0y^{\frac{n}{2}}-a_1y^{\frac{n}{4}}z,y^{\frac{l-n+1}{4}}z>,$ where deg $y=4,$ deg $=n,$ $l-n\equiv 3$(mod $4$), $ m\equiv 3$(mod $4$) and $a_0,a_1 \in \mathbb{Q},$ $a_0=0$ if $n$ is odd or $m<2n~\&~$  $a_1=0$ if $n \not\equiv 0$(mod $4$) or $l<2n.$ 
			
		\end{enumerate}
	\end{theorem}
	\begin{proof}
		It is clear that the $E_2$-term of the Leray-Serre spectral sequence associated to the Borel
		fibration $ X \stackrel{i} \hookrightarrow X_G \stackrel{\pi} \rightarrow B_G$ is given by $E^{k,i}_2= H^{k}(B_G) \otimes H^i(X),$ and $E_2^{*,*} \neq E_{\infty}^{*,*}.$ The least $r$ for which $d_r\neq 0$ must be one of the following: (i) $m-n+1,~$ (ii) $l-m+1$ or~ (iii)~ $l-n+1.$\\
		\textbf{Case(i)} Let $d_{m-n+1} \neq 0.$ Then we must have $m-n \equiv 3$(mod $4$) and $d_{m-n+1}(1 \otimes b)= t^{\frac{m-n+1}{4}} \otimes a $ when $m<l$ or \{$d_{m-n+1}(1 \otimes b)= t^{\frac{m-n+1}{4}} \otimes a $ or $d_{m-n+1}(1 \otimes c)= t^{\frac{m-n+1}{4}} \otimes a $ or $d_{m-n+1}(1 \otimes b)= d_{m-n+1}(1 \otimes c)= t^{\frac{m-n+1}{4}} \otimes a$\} when $m=l.$   By Proposition \ref{prop 4.5}, $d_{l+1} \neq 0.$ So,  we must have $l\equiv 3$(mod $4$) and $d_{l+1}(1 \otimes c)=t^{\frac{l+1}{4}}\otimes 1,$ when $m<l$ or  \{$d_{l+1}(1 \otimes c)=t^{\frac{l+1}{4}}\otimes 1$ or $d_{l+1}(1 \otimes b)=t^{\frac{l+1}{4}}\otimes 1$ or $d_{l+1}(1 \otimes (b+c))=t^{\frac{l+1}{4}}\otimes 1$  \} when $m=l.$  Consequently, $E_{l+2}^{*,*}\cong E^{*,*}_{\infty}$ and hence $E^{k,0}_{\infty}= \mathbb{Q}, 0 \leq k \equiv 0(\mbox{mod} ~4)\leq l; E^{k,n}_{\infty}= \mathbb{Q},0 \leq k\equiv 0(\mbox{mod} ~4) < m-n;$ and trivial otherwise.  For $n \equiv 0$(mod ~4), we get	
		\[
		H^j(X_G)=
		\begin{cases}
			\mathbb{Q}  & 0 \leq j \equiv 0(\mbox{mod} ~4)<  n~\&~m<j\equiv 0(\mbox{mod} ~4) < l \\
			\mathbb{Q} \oplus \mathbb{Q}  & n\leq j \equiv 0(\mbox{mod} ~4) < m\\
			0 & \text{otherwise}
		\end{cases}
		\] 
		For $n \equiv d(\mbox{mod} ~4),$  $1 \leq d \leq 3,$ we get 	
		\[
		H^j(X_G)=
		\begin{cases}
			\mathbb{Q} & 0 \leq j \equiv 0(\mbox{mod} ~4)<  n,n\leq j\equiv d(\mbox{mod} ~4) < m ~\&~ m<j\equiv 0(\mbox{mod}~ 4)<l \\
			0 & \text{otherwise}
		\end{cases}
		\] 
		Let $y \in E_{\infty}^{4,0}$ and $u \in E_{\infty}^{0,n} $  be elements corresponding to the permanent cocyles   $t\otimes 1$ and $1\otimes a,$  respectively. Clearly, $y^{\frac{l+1}{4}}=u^2=y^{\frac{m-n+1}{4}}u=0.$ Then the total complex is given by Tot$E_{\infty}^{*,*}=\mathbb{Q}[y,u]/<y^{\frac{l+1}{4}},u^2,y^{\frac{m-n+1}{4}}u>$ where deg $y=4$ and deg $u=n.$ We get an element $z \in H^n(X_G)$ corresponding to $u$ such that $i^*(z)=a.$ So, we have $z^2=a_0y^{\frac{n}{2}}+a_1y^{\frac{n}{4}}z $ and $ y^{\frac{m-n+1}{4}}z=a_2y^{\frac{m+1}{4}},$ $a_0,a_1,a_2 \in \mathbb{Q},$  where $a_0=0$ if $n$ is odd or $l<2n,$ $a_1=0$ if $n\not\equiv 0$(mod $4$) or $m<2n~\&~a_2=0$ if $m+1 \not\equiv 0$(mod $4$) or $m=l.$ Thus the cohomology ring $H^*(X/G)$ is given by $$\mathbb{Q}[y,z]/<y^{\frac{l+1}{4}},z^2-a_0y^{\frac{n}{2}}-a_1y^{\frac{n}{4}}z,zy^{\frac{m-n+1}{4}}-a_2y^{\frac{m+1}{4}}>$$ where deg $y=4,$  deg $z=n,$  $m-n\equiv 3$ (mod $4$) and $ l \equiv 3$(mod $4$).  This realizes possibility (1).\\
		\textbf{Case(ii)} Let $d_{l-m+1}\neq 0.$ Then  $l-m \equiv 3$(mod $4$) and  $d_{l-m+1}(1 \otimes c)= t^{\frac{l-m+1}{4}}\otimes b$ when $n<m$ or \{$d_{l-m+1}(1 \otimes c)= t^{\frac{l-m+1}{4}} \otimes b$  or $d_{l-m+1}(1 \otimes c)= t^{\frac{l-m+1}{4}} \otimes a$\} when $n=m.$ As $G$ acts freely on $X,$ $l\leq n+m$ and  $d_{n+1}\neq 0.$ Clearly,  $n\equiv 3 $(mod $4$) and $d_{n+1}(1 \otimes a)= t^{\frac{n+1}{4}} \otimes 1$ when $n<m$ or \{$d_{n+1}(1 \otimes a)=t^{\frac{n+1}{4}} \otimes 1$ or $d_{n+1}(1 \otimes b)=t^{\frac{n+1}{4}} \otimes 1$\} when $n=m.$ Thus,  $E_{n+2}^{*,*}\cong E_{\infty}^{*,*},$ and hence $E_{\infty}^{k,0}= \mathbb{Q}$ for $0 \leq k \equiv 0(\mbox{mod} ~4)<n;$  $E_{\infty}^{k,m}=\mathbb{Q}$ for $ 0 \leq k \equiv 0(\mbox{mod} ~4)< l-m;$ and trivial otherwise. Consequently,
		\[
		H^j(X_G)=
		\begin{cases}
			\mathbb{Q}  & 0 \leq j \equiv 0(\mbox{mod} ~4)<  n~\&~m \leq j\equiv d(\mbox{mod} ~4) < l, 0 \leq d \leq 4 \\
			0 & \text{otherwise}
		\end{cases}
		\] 
		Let $y \in E_{\infty}^{4,0}$ and  $u \in E_{\infty}^{0,m} $  be elements corresponding to the permanent cocyles   $t\otimes 1$ and $1\otimes b$  respectively. Clearly, $y^{\frac{n+1}{4}}=u^2=y^{\frac{l-m+1}{4}}u=0.$ Then the total complex is given by Tot$E_{\infty}^{*,*}=\mathbb{Q}[y,u]/<y^{\frac{n+1}{4}},u^2,y^{\frac{l-m+1}{4}}u>$ where deg $y=4$ and deg $u=m.$ We get an element $z \in H^n(X_G)$ corresponding to $u$ such that $i^*(z)=a$ or $b.$  Thus the cohomology ring $H^*(X/G)$ is given by 
		$$\mathbb{Q}[y,z]/<y^{\frac{n+1}{4}},z^2,y^{\frac{l-m+1}{4}}z>$$ where deg $y=4,$ deg $z=m,$  $l-m\equiv 3$ (mod $4$) and $ n \equiv 3$(mod $4$). This realizes possibility (2).\\
		\textbf{Case(iii)} Let $d_{l-n+1}\neq 0.$ Then $l-n \equiv 3$(mod $4$) and we consider $n<m<l.$ So, we have   $d_{l-n+1}(1 \otimes c)=t^{\frac{l-n+1}{4}}\otimes 1.$  As $G$ acts freely on $X,$   $d_{m+1} \neq 0.$ Thus, $ m \equiv 3$(mod $4$) and   $d_{m+1}(1 \otimes b) =t^{\frac{m+1}{4}} \otimes 1.$  Thus, $E_{m+2}^{*,*}= E_{\infty}^{*,*},$ and hence $E_{\infty}^{k,0}= \mathbb{Q}$ for $0 \leq k \equiv 0(\mbox{mod} ~4)<m;$ $E_{\infty}^{k,n}=\mathbb{Q}$ for $ 0 \leq k \equiv 0(\mbox{mod} ~4)< l-n;$ and trivial otherwise. For $n \equiv 0$(mod $4$), we get	
		\[
		H^j(X_G)=
		\begin{cases}
			\mathbb{Q}  & 0 \leq j \equiv 0(\mbox{mod} ~4)<  n~\&~m<j\equiv 0(\mbox{mod} ~4) < l \\
			\mathbb{Q} \oplus \mathbb{Q}  & n\leq j \equiv 0(\mbox{mod} ~4) < m\\
			0 & \text{otherwise}
		\end{cases}
		\] 
		For $n \equiv d(\mbox{mod} ~4),1 \leq d \leq 3,$  we get 	
		\[
		H^j(X_G)=
		\begin{cases}
			\mathbb{Q}  & 0 \leq j \equiv 0(\mbox{mod} ~4)<  m~\&~n\leq j\equiv d(\mbox{mod}~4) < l \\
			0 & \text{otherwise}
		\end{cases}
		\] 
		Let $y \in E_{\infty}^{4,0}$ and $u \in E_{\infty}^{0,n} $  be elements corresponding to the permanent cocyles   $t\otimes 1$  and $1\otimes a,$ respectively. Clearly, $y^{\frac{m+1}{4}}=u^2=y^{\frac{l-n+1}{4}}u=0.$ Then the total complex is given by Tot$E_{\infty}^{*,*}=\mathbb{Q}[y,u]/<y^{\frac{m+1}{4}},u^2,y^{\frac{l-n+1}{4}}u>,$ where deg $y=4$ and deg $u=n.$ We get an element $z \in H^n(X_G)$ corresponding to $u$ such that $i^*(z)=a.$ We have  $z^2=a_0y^{\frac{n}{2}}+a_1y^{\frac{n}{4}}z,$ $a_0,a_1 \in \mathbb{Q},$ where ~$a_0=0$ if $n$ is odd or $m<2n$  and $a_1=0$ if $n \not\equiv 0$(mod $4$) or $l<2n.$ Thus the cohomology ring $H^*(X/G)$ is given by $$\mathbb{Q}[y,z]/<y^{\frac{l+1}{4}},z^2-a_0y^{\frac{n}{2}}-a_1y^{\frac{n}{4}}z,y^{\frac{l-n+1}{4}}z>$$ where deg $y=4,$  deg $z=n,$ $l-n\equiv 3$ (mod $4$) and $ m \equiv 3$(mod $4$).  This realizes possibility (3). 
	\end{proof}
	\noindent Next, we realize some possibilities of our main Theorems of this section.
	\begin{example}\label{4.8}
		Let $W$ and $Z$ be free  $G$-spaces and $A=W\cap Z $ be a closed  subspace of  $W.$ If $A$ is invariant under $G$ then  the adjunction space $W \cup_{A} Z$  obtained by attaching $W$ and $Z$ along with identity map $i: A \rightarrow A$ is also a free $G$-space. It is easy to observe that $(W\cup_A Z)/G = W/G \cup_{A/G} Z/G.$ Consider $W=\mathbb{S}^{4m+3}, Z=\mathbb{S}^{4l+3}$ and $A=\mathbb{S}^{4n+3} = \mathbb{S}^{4m+3} \cap \mathbb{S}^{4l+3}$ in
		Example \ref{example 4.1}. We have a free action of $G=\mathbb{S}^3$  on $Y = W \cup_{A} Z,$  and hence the orbit space $Y/G$  is $ \mathbb{H}P^{m} \cup_{\mathbb{H}P^{n}} \mathbb{H}P^{l}.$ We have seen that $Y \sim_\mathbb{Q}\mathbb{S}^{4n+4}\vee\mathbb{S}^{4m+3}\vee\mathbb{S}^{4l+3},$ where $1 \leq n  <m <l.$ \\ 
		\noindent We use the Gysin  sequence to compute the cohomology algebra of the orbit space $Y/G.$  As $Y/G$ is connected, $H^0(Y/G) \cong \mathbb{Q}.$ Clearly, $H^{i}(Y/G) \cong H^i(Y)$ for $0 \leq i \leq 2.$ Also, we get    $\cup : H^{i-3}(Y/G) \rightarrow H^{i+1}(Y/G)$ are isomorphisms  for $3 \leq i  < 4n+4,$ and  hence  $H^i(Y/G) \cong \mathbb{Q}~ \forall~ 0 \leq i \equiv 0 (\mbox{mod}~ 4) < 4n+4,$ and $H^i(Y/G)$ are trivial homomorphisms  $  \forall~ 0 \leq i \not\equiv 0 (\mbox{mod}~ 4) < 4n+4.$ Let  $y \in H^4(Y/G)$ be the characteristic class of $3$-sphere bundle   $\mathbb{S}^3 \hookrightarrow Y  \stackrel{\pi} \rightarrow Y/\mathbb{S}^3.$ For  the fibre bundle $G\hookrightarrow \mathbb{S}^{4l+3}  \stackrel{\pi'}\rightarrow \mathbb{H}P^l,$ the inclusion map $(i',i): (\mathbb{S}^{4l+3},\pi', \mathbb{H}P^{l}) \rightarrow (Y,\pi,Y/G)$ is a bundle morphism. So, we get $y^i \neq 0$ for $1 \leq i \leq l.$ Clearly, $y^{l+1}=0.$  By the exactness of the sequence we get $H^{4n+4} (Y/G)\cong \mathbb{Q} \oplus \mathbb{Q} ,$ $H^i(Y/G) \cong H^{i+4}(Y/G) \cong \mathbb{Q}\oplus \mathbb{Q}$ for $4n+4 \leq i \equiv 0 (\mbox{mod}~ 4) < 4m+3$ and $H^i(Y/G) = 0$ for $4n+4 \leq i \not\equiv 0 (\mbox{mod}~ 4) < 4m+3.$  Note that $H^{4m+3}(Y/G)$ must be trivial otherwise we get $H^i(Y/G) \neq 0$ for $i>4l+3,$ which  contradict Proposition \ref{prop 4.5}. Further, $H^{4m+4}(Y/G) \cong \mathbb{Q}$ and  $H^i(Y/G) \cong H^{i+4}(Y/G)$ for $4m+4 \leq i \equiv 0 (\mbox{mod}~ 4) < 4l+3.$ Let $z$ be  generator of $H^{4n+4} (Y/G)$ such that $\pi^*(z)=a.$ So, $y^{{m-n}}z=a_2y^{{m+1}}, a_2 \in \mathbb{Q}.$  Clearly,  $ z^2=a_0y^{{2n+2}}+a_1y^{{n+1}}z$ where $a_0,a_1 \in \mathbb{Q}$ and $a_0=0$ if $4l+3<8n+8,a_1=0$ if $4m+3<8n+8.$ Thus, the cohomology ring $H^*(Y/G) \cong \mathbb{Q}[y,z]/<y^{{l+1}},z^2-a_0y^{{2n+2}}-a_1y^{{n+1}}z,zy^{{m-n}}-a_2y^{{m+1}}>$ where deg $y=4$ and deg $z=4n+4.$ This realizes possibility (1) of Theorem \ref{thm 4.8}.	\end{example}
	\begin{example}
		As in Example \ref{4.8}, we can define free actions of (i) $G=\mathbb{Z}_2$ on $Y \sim_{\mathbb{Z}_2} \mathbb{S}^{n+1} \vee \mathbb{S}^m \vee \mathbb{S}^l,$  $m \neq n+2 ~\&~ l \neq m+1,$ (ii) $G=\mathbb{Z}_p,~ p>2$ a prime  on $Y \sim_{\mathbb{Z}_p} \mathbb{S}^{2n+2} \vee \mathbb{S}^{2m+1} \vee \mathbb{S}^{2l+1},$ and (iii) $G=\mathbb{S}^1$ on $Y \sim_{\mathbb{Q}} \mathbb{S}^{2n+2} \vee \mathbb{S}^{2m+1} \vee \mathbb{S}^{2l+1}.$ The orbit spaces $Y/G$  of these actions are $ \mathbb{R}P^m \cup_{\mathbb{R}P^n} \mathbb{R}P^l,$ $ L^m \cup_{L^n} L^l$ and $ \mathbb{C}P^m \cup_{\mathbb{C}P^n} \mathbb{C}P^l,$ respectively where $L^n$ denotes the  lens space.  This realizes possibility (1) of Theorems \ref{thm 4.4}, \ref{zp action} and
		\ref{thm 4.7}. 
	\end{example}
	%	\begin{example}
	%	Consider  a free action of  $G=\mathbb{Z}_2$ on adjunction space $Y=\mathbb{S}^n \cup_{\mathbb{S}^{n-1}} \mathbb{S}^l,~ l\neq 2n,$  having mod $2$ cohomology of $\mathbb{S}^n \vee \mathbb{S}^n \vee \mathbb{S}^l$ where $\pi_1(B_G)$ acts nontrivially on $H^n(\mathbb{S}^n \vee \mathbb{S}^n \vee \mathbb{S}^l; \mathbb{Z}_2).$ By  Example \ref{4.8}, the orbit space $Y/G=\mathbb{R}P^n\cup_{\mathbb{R}P^{n-1}} \mathbb{R}P^l.$ Using the Gysin sequence it is easy to see that $H^*(Y/G)=\mathbb{Z}_2 [y,z]/<y^{l+1},z^2+a_0y^{2n},yz+a_1y^{n+1}>,$ where deg $y=1,$ deg $z=n,$ and $ a_0=0$ if $l<2n.$ This  realizes case (1) of Theorem \ref{4.2}.
	%	\end{example}
	\section{Applications}
	\noindent  In this section,  we derive the Borsuk-Ulam type results for free actions of $G=\mathbb{S}^d,~ d=0,1,3$ on $X \sim_R \mathbb{S}^n \vee \mathbb{S}^m  \vee {S}^l, $ $1 \leq n \leq m \leq l ,$  where    $R=\mathbb{Z}_2$ for $G=\mathbb{Z}_2$ and $R=\mathbb{Q}$ for $G=\mathbb{S}^d,d=1,3.$ We determine the nonexistence of $G$-equivariant maps between  $ X$ and $ \mathbb{S}^{(d+1)k+d}$ where $\mathbb{S}^{(d+1)k+d}$ equipped with  standard free actions of $G.$\\
	\indent Recall that \cite{floyd, anju} the index (respectively, co-index) of a $G$-space $X$ is the greatest integer $k$ (respectively, the lowest integer $k$)  such that there exists a $G$-equivariant map $\mathbb{S}^{(d+1)k+d} \rightarrow X$ (respectively, $ X \rightarrow  \mathbb{S}^{(d+1)k+d}$).\\
	\indent	By  Theorems \ref{thm 4.4}, \ref{thm 4.7} and \ref{thm 4.8}, it is clear that  the largest integer  $s$ such that $w^s \neq 0$ is for some $s\in \{n,m,l\},$ where $w \in H^{d+1}(X/G)$ is the  characteristic class of the principle $G$-bundle $G \hookrightarrow X \rightarrow X/G.$ We have index$(X)\leq s$ \cite{floyd, anju}. So, we have
	\begin{theorem}
		Let $G=\mathbb{S}^d, ~d=0,1,3$ act freely on $X \sim_R \mathbb{S}^n \vee \mathbb{S}^m  \vee {S}^l,$ where $ 1 \leq n \leq m \leq l$ and $R= \mathbb{Z}_2$ or $
		\mathbb{Q}.$ Then there does not exist $G$-equivariant map from $\mathbb{S}^{(d+1)k+d} \rightarrow X,$ for $k > s,$ where $s$ is one of $n,m$ or $l.$
	\end{theorem}
	Recall that the Volovikov's index $i(X)$ is the smallest $r$ such that $d_r: E_r^{k-r,r-1} \rightarrow E_r^{k,0}$ is nontrivial for some $k,$ in the Leray-Serre spectral sequence of the fibration  $ X \stackrel{i} \hookrightarrow X_G \stackrel{\pi} \rightarrow B_G.$ Form Theorems \ref{thm 4.4}, \ref{thm 4.7} and \ref{thm 4.8}, we have $i(X)$ is  $ n+1 ,m+1$ or $l+1.$ By taking $Y=\mathbb{S}^{(d+1)k+d}$ in Theorem 1.1 \cite{co}, we have
	\begin{theorem}
		Let $G=\mathbb{S}^d,~d=0,1,3$ act freely on a finite CW-complex $X \sim_R \mathbb{S}^n \vee \mathbb{S}^m \vee \mathbb{S}^l, 1 \leq n\leq m \leq l .$ Then there is no $G$-equivariant map $f: X \rightarrow \mathbb{S}^{(d+1)k+d}$  if $(d+1)k+d< i(X)-1,k \geq 1,$ where $i(X)=n+1,m+1$ or $l+1.$ 
		
	\end{theorem}

	\bibliographystyle{plain}

\begin{thebibliography}{100}
		\bibitem{bredon}  Bredon, G. E.: Introduction to Compact Transformation Groups. New York, USA Academic Press, (1972)
		
		
		\bibitem{Bredon}  Bredon, G. E.: The cohomology ring structure of a fixed point set.  Ann. of Math. \textbf{80}, 524-537(1964)
		\bibitem{c} Chang, T.,  Comenetz,  M.: Group actions on cohomology projective planes and products of spheres.  Q. J. Math. \textbf{28}, (1977) 
		\bibitem{co} Coelho, F. R. C., Mattos,  D. de,  Santos,  E. L. dos:  On the existence of G-equivariant maps.  Bull. Braz. Math. Soc. \textbf{43}, 407-421(2012)
		\bibitem{floyd}  Conner, P. E.,  Floyd, E. E.:  Fixed point free involutions and equivariant maps.  Bull. Amer. Math. Soc. \textbf{66}, 416-441(1960)
		\bibitem{o}  Dotzel, R. M., Singh, T. B.: Cohomology ring of the orbit space of certain free $\mathbb{Z}_ p$ actions. Proc. Amer. Math. Soc. \textbf{123}, 3581-3585(1995)
		\bibitem{f}  Dotzel, R. M.,  Singh, T. B.: $\mathbb{Z}_ p$ actions on spaces of cohomology type (a,0).  Proc. Amer. Math. Soc. \textbf{113}, (1991)
		\bibitem{orbit space}  Dotzel, R. M.,  Singh, T. B.,  Tripathi, S. P.:  The cohomology rings of the orbit spaces of free transformation groups of the product of two spheres.   Proc. Amer. Math. Soc. \textbf{129}, 921-930(2000)
		\bibitem{dey} Dey, P., Singh,  M.:  Free actions of some compact groups on Milnor manifolds. Glasg. Math. J. \textbf{61}, 727-742(2019)
		\bibitem{allen} Hatcher,  A.: Algebraic Topology. Cambridge University Press, Cambridge, (2000) 
		\bibitem{monika} Jain,  M.: A study of homology of fixed point set. M.Phil Dissertation, University of Delhi, (1997)
		\bibitem{anju} Kumari,  A., Singh, H. K.:   Fixed point free actions of spheres and equivarient maps.   Topology Appl. \textbf{305}, 107886(2022) 
		\bibitem{kaur}  Kaur, J.,   Singh, H. K.:  On the existence of free action of $\mathbb{S}^3$ on certain finitistic mod p cohomology spaces.  J. Indian Math. Soc. \textbf{82},  97–106(2015)
		\bibitem{mac}   McCleary, J.: A user’s guide to spectral sequences. Cambridge Studies in Advanced Mathematics, 
		Cambridge University Press, IInd edition, \textbf{58}, (2001)
		\bibitem{Mattos} Morita,  A. M. M.,  Mattos, D. De, Pergher, P. L. Q.: The cohomology ring of orbit spaces of free $\mathbb{Z}_2$-actions on some 
		Dold manifolds. Bull. Aust. Math. Soc. \textbf{97}, 340-348(2018)
		\bibitem{s}   Smith, P. A.: Fixed-point theorems for periodic transformations. Amer. J. Math. \textbf{63}, 1-8(1941)
		\bibitem {hemant sir} Pergher, P. L. Q.,   Singh,  H. K.,  Singh, T. B.: On $\mathbb{Z}_2$ and $S^1$ free actions on spaces of cohomology type (a,b).   Houston J. Math. \textbf{36(1)}, 137-146(2010)
		\bibitem{singh} Singh,  M.: $\mathbb{Z}_2$ actions on complexes with three non-trivial cells.  Topology Appl. \textbf{155}, 965-971(2008)
		\bibitem{m circle} Singh,  M.:  Fixed points of circle actions on spaces with rational cohomology of $S^ n \vee
		S^{2n} \vee S ^{3n}$ or $P ^2 (n) \vee S ^{3n}.$  Arch. Math. \textbf{92}, 174-183(2009)
		
		\bibitem{j1} Su, J. C.:  Periodic transformations on the product of two spheres.  Trans. Amer.
			Math. Soc. \textbf{112}, 369-380(1964)
		\bibitem{tom}  Dieck, T. Tom:  Transformation Groups. De Gruyter Studies in Mathematics,  \textbf{8}, Walter de Gruyter \& Co., Berlin, (1987)
	\end{thebibliography}

\end{document}